\documentclass[12pt, reqno]{amsart}

\usepackage{amsmath, amssymb, amsfonts, amsbsy, amsthm, latexsym}
\usepackage{graphicx}
\usepackage{verbatim}
\usepackage{color}
\usepackage{xy}
\usepackage{pstricks, pst-node}
\usepackage{caption}
\usepackage{subcaption}

\newtheorem{theorem}{Theorem}[section]

\newtheorem{lemma}[theorem]{Lemma}
\newtheorem{corollary}[theorem]{Corollary}

\newtheorem{question}[theorem]{Question}

\theoremstyle{definition}
\newtheorem{definition}[theorem]{Definition}
\newtheorem{example}[theorem]{Example}

\newtheorem{remark}[theorem]{Remark}

\numberwithin{equation}{section}

\newcommand{\C}{{\mathbb{C}}}

\newcommand{\CHI}{\hbox{\raise .4ex \hbox{$\chi$}}}

\newcommand{\R}{{\mathbb{R}}}
\newcommand{\Sb}{\mathbb{S}}

\newcommand{\E}{{\mathbb{E}}}

\begin{document}

\title{Randomized subspace actions and fusion frames}

\author{Xuemei Chen}
\address{University of Missouri, Department of Mathematics, Columbia, MO 65211}
\email{chenxuem@missouri.edu}


\author{Alexander M. Powell}
\address{Vanderbilt University, Department of Mathematics,
Nashville, TN 37240}
\email{alexander.m.powell@vanderbilt.edu}
\thanks{The authors were supported in part by NSF DMS Grant 1211687.}


\date{\today}


\begin{abstract}
A randomized subspace action algorithm is investigated for fusion frame signal recovery problems and for the problem
of recovering a signal from projections onto random subspaces.
It is noted that Kaczmarz bounds provide upper bounds on the algorithm's error moments.  The main question
of which probability distributions on a random subspace lead to provably fast convergence is addressed.   In particular, it is proven which distributions give minimal Kaczmarz bounds, and hence give best control on error moment upper bounds arising from Kaczmarz bounds.
Uniqueness of the optimal distributions is also addressed.
\end{abstract}

\maketitle

\section{Introduction}

Fusion frames are a mathematical tool for distributed signal  processing and data fusion.
Complexity and computational constraints in high dimensional problems can limit the amount of global
processing that is possible and often require approaches that are built up from local processing.
For example, in wireless sensor networks, physical constraints on sensors mean that
global processing is typically organized through processors on local subnetworks.
Mathematically, a fusion frame provides global signal representations by fusing together projections
on local subspaces.  

Fusion frames were introduced in \cite{CK04} as a generalization of the classical notion of frames.
Let $H$ be a Hilbert space and let $I$ be an at most countable index set.
A collection $\{\varphi_n\}_{n\in I} \subset H$ is a {\em frame} for $H$ 
with {\em frame bounds} $0<A \leq B < \infty$ if 
\begin{equation} \label{frame-def}
\forall x \in H, \ \ \ A \|x\|^2 \leq \sum_{n \in I} | \langle x, \varphi_n \rangle |^2 \leq B \|x\|^2.
\end{equation}
The frame inequality \eqref{frame-def} ensures that the frame coefficients $\langle x, \varphi_n \rangle$
stably encode $x\in H$, and that there exists a (possibly nonunique)
dual frame $\{\psi_n\}_{n\in I} \subset H$ such that the following unconditionally convergent 
frame expansions hold
\begin{equation} \label{frame-exp}
\forall x \in H, \ \ \ x = \sum_{n\in I} \langle x, \varphi_n \rangle \psi_n 
= \sum_{n\in I} \langle x, \psi_n \rangle \varphi_n. 
\end{equation}
An important aspect of frame theory is that frames can be redundant or overcomplete.  Redundancy endows the frame expansions \eqref{frame-exp} with robustness properties that are useful in applications such as multiple description coding \cite{GKV98}, transmission of data over erasure channels \cite{GKK01, GKV99}, and quantization \cite{BPY}.  
See \cite{Cas, CK10} for an introduction to frame theory and its applications.

Fusion frames take the idea \eqref{frame-def} one step further and replace
the scalar frame coefficients  $\langle x, \varphi_n \rangle$ 
by projections onto a redundant collection of subspaces.
Let $\{W_n\}_{n\in I}$ be a collection of closed subspaces of $H$
and let $\{v_n\}_{n\in I} \subset (0,\infty)$ be a collection of positive weights.
The collection $\{(W_n, v_n)\}_{n\in I}$ is said to be a {\em fusion frame} for $H$ 
with {\em fusion frame bounds} $0<A \leq B < \infty$  if
\begin{equation} \label{ff-def}
\forall x \in H, \ \ \ A \|x\|^2 \leq \sum_{n\in I}  v_n^2 \| P_{W_n} (x)\|^2 \leq B \|x\|^2,
\end{equation}
where $P_W$ denotes the orthogonal projection onto a subspace $W \subset H$.  If $A=B$ then the fusion frame is said to be {\em tight}.
Note that in the special case when each $W_n$ is a one-dimensional subspace with $W_n = {\rm span} ( \varphi_n )$ and $v_n = \|\varphi_n\|$, the fusion frame inequality \eqref{ff-def} reduces to the statement \eqref{frame-def} that $\{\varphi_n\}_{n\in I}$ is a frame for $H$.  
For further background on fusion frames see \cite{CF09, CFMWZ, B13, CCL}.

If $\{(W_n, v_n)\}_{n\in I}$ is a fusion frame for $H$ then the associated {\em fusion frame operator} $S:H \to H$ is defined by
$$S(x) = \sum_{n\in I} v_n^2 P_{W_n}(x).$$
It is known, e.g., \cite{CK04, CKL08}, that $S$ is a positive invertible operator  and that each
$x\in H$ can be recovered from its fusion frame projections $y_n = P_{W_n}(x) \in H,$ $n \in I,$ by  
\begin{equation*}
x = (S^{-1}\circ S)(x)=\sum_{n\in I} v_n^2 (S^{-1} \circ P_{W_n})(x) =  \sum_{n\in I} v_n^2 S^{-1} (y_n).
\end{equation*}
Practical inversion of the fusion frame operator $S$ and, more generally, reconstructing $x\in \R^d$ from the projections $y_n = P_{W_n}(x)$, can be computationally intensive.

The following extension of the classical frame algorithm gives an iterative way to recover
$x\in H$ from its fusion frame projections $y_n = P_{W_n}(x)$, $n\in I$.
Given an arbitrary initial estimate $x_0\in H$, the algorithm
produces estimates $x_n \in H$ of $x$ from the projections $\{y_n\}_{n\in I}$ by iterating
for $n \geq 1$
\begin{equation} \label{ffalg}
x_n = x_{n-1} + \frac{2}{A+B} \left[ \left( \sum_{j\in I} v_j^2 y_j \right) - S(x_{n-1}) \right]
=x_{n-1} + \frac{2}{A+B} S(x-x_{n-1}).
\end{equation}
Similar to the situation for frames, it was shown in \cite{CKL08} that the fusion frame algorithm \eqref{ffalg} satisfies
$\lim_{n\to \infty} x_n =x$ with 
\begin{equation*}
\| x - x_n\| \leq \left( \frac{B-A}{B+A} \right)^n \|x\|.
\end{equation*}
For other iterative approaches to fusion frame reconstruction
and a discussion of local versus global aspects of reconstruction see \cite{CKL08, KA07, O}.

\subsection{Recovery by iterative subspace actions}

This article will analyze an algorithm, motivated by row-action methods, for recovering a signal from projections onto subspaces. 
This algorithm can be directly used as a fusion frame recovery algorithm.
We restrict our attention to the finite dimensional setting and 
let $\mathbb{H}^d$ be the $d$-dimensional Hilbert space with ${\mathbb{H}} = \R$ or $\C$.  
Suppose that $\{W_n\}_{n=1}^{\infty}$ is
a given collection of subspaces of $\mathbb{H}^d$. The goal is to recover
$x \in \mathbb{H}^d$ from the set of projections $y_n = P_{W_n}(x)$, with $n \geq1$.

We focus on the following iterative algorithm.  Let $x_0 \in \mathbb{H}^d$ be an arbitrary initial
estimate.  Produce updated estimates $x_n \in \mathbb{H}^d$ for $x\in \mathbb{H}^d$ by running the following iteration for $n\geq 1$
\begin{equation} \label{ff-kacz}
x_n=x_{n-1}+y_n - P_{W_n}(x_{n-1}).
\end{equation}
The algorithm \eqref{ff-kacz} dates at least back to \cite{E80} in the context of block-iterative methods for linear equations.
Note that in contrast to \eqref{ffalg},  each iteration of \eqref{ff-kacz} only acts on a single subspace 
$W_n$ and a single measurement $y_n$, and for this reason, we shall refer to \eqref{ff-kacz} as a {\em subspace action method}. 

Algorithms of the type \eqref{ff-kacz} have a long history.  
Geometrically, \eqref{ff-kacz} is an example of a projection onto convex sets (POCS) algorithm,
and in particular $x_n$ is simply the orthogonal projection of $x_{n-1}$ onto the convex set 
$\{ u \in \mathbb{H}^d : P_{W_n}(u) = y_n \}$.  The algorithm \eqref{ff-kacz} also falls into the class of block iterative methods related to the Kaczmarz algorithm, e.g., \cite{E80, NT}.  For example, if each $W_n = {\rm span}(\varphi_n)$ is one dimensional, then $P_{W_n}(x) = \varphi_n \langle x, \varphi_n \rangle / \|\varphi_n\|^2$ and \eqref{ff-kacz}
reduces to the familiar Kaczmarz algorithm for recovering $x\in \R^d$ from the linear measurements $y_n = \langle x, \varphi_n \rangle, n \geq 1$,
\begin{equation} \label{kacz-algorithm}
x_n = x_{n-1} + \frac{y_n - \langle x_{n-1}, \varphi_n \rangle}{\|\varphi_n\|^2} \varphi_n,
\end{equation}
for example, see \cite{K, SV09}.  Randomized versions of this algorithm were recently studied for two-dimensional subspaces in \cite{NW} and for general subspaces in \cite{NT}.

The following examples illustrate different scenarios for implementing and analyzing the subspace action method \eqref{ff-kacz}.
The underlying algorithm is the same in each example, but there are illustrative differences in perspective regarding how the algorithm is applied.  

\begin{example} [Cyclic subspace actions for fusion frames]  \label{example-cyclic-ff}
Let $\{(U_n, v_n)\}_{n=1}^N$ be a  fusion frame for $\mathbb{H}^d$
and suppose that one wants to recover $x\in \mathbb{H}^d$ from the measurements $y_n = P_{U_n}(x)$, $1\leq n \leq N$.
Define the infinite collection of subspaces $\{W_n\}_{n=1}^{\infty}$ in $\mathbb{H}^d$ by 
$W_n = U_{p(n)}$, where  $p(n) \in \{1, 2, \cdots, N\}$ satisfies $p(n) \equiv n$ \ (modulo $N$).
When the algorithm \eqref{ff-kacz} acts on $\{W_n\}_{n=1}^\infty$, this corresponds to cyclicly iterating on the finite collection of deterministic fusion frame subspaces
$\{U_n\}_{n=1}^N$.  This is comparable to the classical Kacmarz algorithm with cyclic control in \cite{K}.
\end{example}

\begin{example} [Randomized subspace actions for fusion frames] \label{example-randomized-ff}
Let $\{(U_n, v_n)\}_{n=1}^N$  be a fusion frame for $\mathbb{H}^d$
and suppose that one wants to recover $x\in \mathbb{H}^d$ from the measurements $y_n = P_{U_n}(x)$, $1\leq n \leq N$.
Define the infinite collection of subspaces $\{W_n\}_{n=1}^{\infty} \subset \mathbb{H}^d$ by $W_n = U_{r(n)},$ where the random numbers $\{r(n)\}_{n=1}^{\infty}$ are drawn independently according to a given probability distribution on $\{1, 2, \cdots, N\}$.  This corresponds to iterating \eqref{ff-kacz} with a random
selection from $\{U_n\}_{n=1}^N$ at each step, cf. Example \ref{rand-det-exam} and Remark \ref{remark-rand-alg-specialized}.
This is comparable to the randomized Kaczmarz algorithm recently studied in \cite{SV09}.
\end{example}

\begin{example} [Subspace actions for i.i.d. random subspaces] \label{example-random-subspace}
Let $W$ be a random subspace of $\mathbb{H}^d$.  Suppose $\{W_n\}_{n=1}^{\infty}$ are independent identically distributed (i.i.d.) versions of $W$, and suppose that
one wants to recover $x\in \mathbb{H}^d$ from the measurements $y_n = P_{W_n}(x)$, $n \geq 1$.  
Here, the subspace actions \eqref{ff-kacz} could be viewed as an online algorithm when the measurements $y_n$ are received in a streaming manner. 
Moreover, this general setting contains, as a special case, Example \ref{example-randomized-ff} on randomized subspace actions for fusion frames.  In particular, 
Example \ref{example-randomized-ff} is the special case when $W$ is a discrete random subspace taking values in a given fusion frame $\{U_n\}_{n=1}^N$.  
Lastly, the generality of this setting is convenient for analysis and contributes to the general understanding of the algorithm.
This is comparable to the framework in \cite{CP}.
\end{example}

It is useful to note differences between the fusion frame algorithm \eqref{ffalg}
and the subspace action method \eqref{ff-kacz} in the fusion frame setting.  The algorithm \eqref{ffalg} requires access to the entire fusion frame system $\{(W_n, v_n)\}_{n\in I}$ and the full set of projections $\{y_n\}_{n\in I}$ at each iteration and also requires some knowledge of the fusion frame bounds $A,B$.  
For high dimensional problems memory issues might pose challenges to storing or using the entire fusion frame system at once.
Moreover, in practice, one might receive access to the fusion frame measurements $y_n = P_{W_n}(x)$ in a streaming manner for which an online algorithm such as \eqref{ff-kacz} might be suitable.

Unfortunately, unlike \eqref{ffalg}, the algorithm \eqref{ff-kacz} is sensitive to the order in which it receives the inputs $y_n = P_{W_n}(x)$, and an inappropriate
ordering of the subspaces can lead to poor performance.  For example, there are instances where the cyclic ordering in Example \ref{example-cyclic-ff} is highly suboptimal.
Recent work on the randomized Kaczmarz algorithm in \cite{SV09} indicates that
a randomized selection of the $y_n=P_{W_n}(x)$, as in Example \ref{example-randomized-ff}, can circumvent such ordering issues and leads to fast convergence.

This paper is motivated by the randomized subspace action method \eqref{ff-kacz} 
for (deterministic) fusion frames $\{(U_n, v_n)\}_{n=1}^N$ in Example \ref{example-randomized-ff}, 
but our analysis focuses on the more general setting of i.i.d. random subspaces $\{W_n\}_{n=1}^{\infty}$ as in Example \ref{example-random-subspace}.
For example, the error bounds in Section \ref{bounds-sec} are stated for the setting of i.i.d. random subspaces, 
but can be specialized to apply to the randomized subspace action method for fusion frames, e.g., see Remark \ref{remark-rand-alg-specialized}.

\subsection{Overview and main results}

The main goal of this work is to provide error bounds for the subspace action algorithm \eqref{ff-kacz} when it is driven by random subspaces $\{W_n\}_{n=1}^{\infty}$,
and to determine which choices of randomization lead to fast convergence in \eqref{ff-kacz}.  The main contributions of this work can be summarized as follows:
\begin{enumerate}
\item As necessary background lemmas, the error bounds of \cite{CP} are extended to the setting of fusion frames and, in particular, they provide bounds on error moments and almost sure convergence rates for the algorithm \eqref{ff-kacz}.
\item We address the following main question.  For which choices of i.i.d. random $k$-dimensional subspaces $\{ W_n\}_{n=1}^{\infty}$ of $\mathbb{H}^d$ does the algorithm \eqref{ff-kacz} converge quickly?  Specifically, we describe random subspaces with minimal Kaczmarz bounds (the Kaczmarz bounds in turn provide upper bounds for the algorithm's error moments).  Uniqueness of minimizers is discussed in special cases.
\end{enumerate}

The paper is organized as follows.  In Section \ref{rand-subspace-sec} we provide necessary background on random subspaces and Kaczmarz bounds.
Section \ref{bounds-sec} provides basic error bounds for the subspace action algorithm \eqref{ff-kacz} when the subspaces $\{ W_n\}_{n=1}^{\infty}$ are i.i.d. versions
of a random subspace $W$; these results are immediate generalizations of the error bounds in \cite{CP} and give upper bounds on error moments of \eqref{ff-kacz}
in terms of Kaczmarz bounds of $W$.  In contrast to later sections, the results in Sections \ref{rand-subspace-sec} and \ref{bounds-sec} do not assume that the $\{W_n\}_{n=1}^{\infty}$ have the same dimensions.

In Section \ref{opt-dist-sec} we address which distributions on a $k$-dimensional random subspace $W$ have minimal Kaczmarz bounds.
These distributions in turn lead to the smallest upper bounds on the error moments obtained in Section \ref{bounds-sec}.
Our first main results, Theorem \ref{thm_argmin} and Corollary \ref{cor_argmin}, show that a distribution achieves the minimal Kaczmarz bound precisely when its Kaczmarz
bound satisfies a probabilistic tightness condition that is analogous to the notion of tight fusion frames. 
In particular, the invariant measure on the Grassmannian $G(k,d)$ is a minimizer for both the Kaczmarz bound of order $s$ and the logarithmic Kaczmarz bound.
In Section \ref{unique-sec} we address uniqueness of the minimizing distributions from Section \ref{opt-dist-sec}; Theorem \ref{thm:unique} shows that the invariant
measure on $G(1,d)$ uniquely minimizes both the logarithmic Kaczmarz bound and the Kaczmarz bound of order $s$ when $0<s<1$ 
(but it is not generally unique when $s=1$).
Section \ref{Kacz-section-at-end} briefly discusses how the main results relate to the Kaczmarz algorithm.
Section \ref{numerics-sec} provides numerical experiments and examples to illustrate our results.

\section{Random subspaces and Kaczmarz bounds} \label{rand-subspace-sec}

We shall primarily be interested in the performance of the algorithm \eqref{ff-kacz} when the 
spaces $\{W_n\}_{n=1}^{\infty} \subset \mathbb{H}^d$ are randomly chosen. In this section,
we provide some necessary background and notation related to random subspaces.

The Grassmannian $G(k,d)=G(k,\mathbb{H}^d)$ is the set of all $k$-dimensional subspaces of $\mathbb{H}^d$.  
Let $\mathcal{G} = \cup_{k=0}^d G(k,d)$ denote the collection of all subspaces of $\mathbb{H}^d$,
let $\mathcal{A}$ be a $\sigma$-algebra of subsets of $\mathcal{G}$, and let $\mathcal{P}$
be a probability measure on $\mathcal{A}$.  Let $\mathbb{S}^{d-1}=\mathbb{S}^{d-1}_{\mathbb{H}}=\{x \in \mathbb{H}^d: \|x\|=1\}$ denote the unit-sphere
in $\mathbb{H}^d$.
To begin, we simply assume
$W$ is a random subspace defined on the probability space $(\mathcal{G}, \mathcal{A}, \mathcal{P})$.
The error bounds in Section \ref{bounds-sec} require the following quantitative notion of how nondegenerate a random subspace is.

\begin{definition} \label{kacz-bnd-def}
Given $s>0$,  the {\em Kaczmarz bound $0 \leq \alpha_s \leq 1$ of order $s$} for a random subspace $W \subset \mathbb{H}^d$ is defined as
$$\alpha_s=\sup_{x \in \mathbb{S}^{d-1}}\left( \mathbb{E} \left[ \left( 1 - \|P_W(x)\|^2 \right)^s \right] \right)^{1/s}.$$
The {\em logarithmic Kaczmarz bound} $0 \leq \alpha_{\log} \leq 1$ for a random subspace $W \subset \mathbb{H}^d$ is defined as
$$\alpha_{\log}=\sup_{x \in \mathbb{S}^{d-1}}{\rm exp} \left( \mathbb{E} \left[ \log \left( 1 - \|P_W(x)\|^2 \right) \right]  \right).$$
In particular, 
\begin{equation} \label{kacz-def}
\forall x \in \mathbb{S}^{d-1}, \ \ \  
\left( \mathbb{E} \left[ \left( 1 - \|P_W(x)\|^2 \right)^s \right] \right)^{1/s} \leq \alpha_s,
\end{equation}
and
\begin{equation} \label{log-kacz-def}
\forall x \in \mathbb{S}^{d-1}, \ \ \  
{\rm exp} \left( \mathbb{E} \left[ \log \left( 1 - \|P_W(x)\|^2 \right) \right]  \right)\leq \alpha_{\log}.
\end{equation}
We shall say that the Kaczmarz bound or logarithmic Kaczmarz bound is {\em tight} if equality
holds in \eqref{kacz-def} or \eqref{log-kacz-def} respectively.
\end{definition}

See \cite{CP} for further discussion of Kaczmarz bounds.
The motivation for the logarithmic Kaczmarz bound \eqref{log-kacz-def} comes from considering the limit as $s\to0$ in \eqref{kacz-def}, see Lemma \ref{log-lim-lemma}.
\begin{lemma}  \label{log-lim-lemma}
Let $X$ be a random variable.  If $s_1 \geq s_2>0$ then
\begin{equation} \label{lyap-ineq}
\left( \mathbb{E} |X|^{s_2} \right)^{1/s_2} \leq \left( \mathbb{E} |X|^{s_1} \right)^{1/s_1}.
\end{equation}
Moreover, if $\mathbb{E}|X|^s<\infty$ for some $s>0$ then
\begin{equation} \label{log-lim}
\inf_{s>0} \left( \mathbb{E}|X|^s \right)^{1/s} =\lim_{s\to 0} \left( \mathbb{E}|X|^s \right)^{1/s}
= {\rm exp} \left( \mathbb{E} \log |X| \right).
\end{equation}
\end{lemma}
The inequality \eqref{lyap-ineq} is known as {\em Lyapunov's inequality}, see page 193 in \cite{S}, and the limit \eqref{log-lim} can, for example, be found on page 71 in \cite{R}.

\begin{example}[Randomized subspace action method for fusion frames] \label{rand-det-exam}

Let $\{(U_n, v_n)\}_{n=1}^N$ be a fusion frame for $\mathbb{H}^d$ with fusion frame bounds $0<A\leq B < \infty$,
and let $W$ be the random subspace defined by
\begin{equation*}\label{W-rand-det}
\forall 1 \leq k \leq N, \ \ \ \Pr [ W = U_k ] = \frac{v_k^2}{\sum_{n=1}^N v_n^2}.
\end{equation*}
This randomization of $\{(U_n,v_n)\}_{n=1}^N$ can be viewed as an instance of Example \ref{example-randomized-ff}.
Then $W$ has a Kaczmarz bound $\alpha_1$ of order $s=1$ which satisfies
\begin{equation*} \label{rand-det-alpha-one}
\alpha_1 \leq 1 - \frac{A}{\sum_{n=1}^N v_n^2}<1.
\end{equation*}
In the special case when $\{(U_n, v_n)\}_{n=1}^N$ is a tight fusion frame then the fusion frame bounds are given by
$A = B= \frac{1}{d} \sum_{n=1}^N v_n^2 \thinspace {\rm dim}(W_n)$, e.g., see Chapter 13 in \cite{CK10},
and in this case the random subspace $W$ has a tight Kaczmarz bound of order $s=1$ given by
$$\alpha_1 = 1 - \frac{\sum_{n=1}^N v_n^2 \thinspace {\rm dim}(W_n)}{d \sum_{n=1}^N v_n^2}.$$
\end{example}

\begin{example} \label{onb-kacz-bnd-example}
Let $\{e_n\}_{n=1}^d$ be an orthonormal basis for $\mathbb{H}^d$, which is also a frame. 
It is easy to verify that a cyclic implementation of the Kaczmarz algorithm \eqref{kacz-algorithm}, as in Example \ref{example-cyclic-ff}, converges to the solution $x$ in at most $d$ iterations. But for the sake of illustration, we shall compute the Kaczmarz bound when we choose the randomized Kaczmarz algorithm with the random subspace $W$ defined by
\begin{equation} \label{onb-kacz-bnd-example-defeq}
\forall 1 \leq n \leq d, \ \ \ {\rm Pr}[ W = {\rm span}(e_n)] = 1/d.
\end{equation}
Then
\begin{equation} \label{onb-exam-eq1}
\left( \mathbb{E} \left( 1 - \|P_W(x)\|^{2}\right)^{s} \right)^{1/s}= \left( \frac{1}{d} \sum_{n=1}^d \left( 1 - |\langle x, e_n \rangle|^{2}\right)^s \right)^{1/s}.
\end{equation}

When $0<s<1$, it can be verified that that supremum of \eqref{onb-exam-eq1} over all $x\in \mathbb{S}^{d-1}$ equals $(1-1/d)$ and is, for example, attained by
$x =d^{-1/2} \sum_{n=1}^d e_n$.  So, when $0<s<1$, $W$ has the Kaczmarz bound of order $s$ given by $\alpha_s = (1 - 1/d)$.

When $1 \leq s < \infty$, it can be verified that that supremum of \eqref{onb-exam-eq1} over all $x\in \mathbb{S}^{d-1}$ equals $(1 - 1/d)^{1/s}$ and is, for example, attained by $x =e_1$.  So, when $1 \leq s < \infty$, $W$ has the Kaczmarz bound of order $s$ given by $\alpha_s =(1 - 1/d)^{1/s}$.

\end{example}

\section{Error bounds} \label{bounds-sec}

This section states error bounds for the subspace action algorithm \eqref{ff-kacz}.  The results of this
section are directly motivated by analogous results in \cite{CP} for the standard Kaczmarz algorithm.
We omit most proofs in this section since they are almost identical to their Kaczmarz 
counterparts in \cite{CP}.

The following basic error bound is proven in the same manner as Proposition 3.1 in \cite{CP}. Notice that \eqref{ff-kacz} implies $x-x_n=P_{W_n^{\perp}}(x-x_{n-1})$, hence \eqref{error-nonorm}.
\begin{lemma}  The error in the algorithm \eqref{ff-kacz} satisfies
\begin{equation} \label{error-nonorm}
(x-x_n) = P_{W_n^{\perp}} \cdots  P_{W_2^{\perp}}  P_{W_1^{\perp}}(x - x_0)
\end{equation}
and
\begin{equation*} 
\|x - x_n\|^2 = \|x- x_{n-1}\|^2 - \| P_{W_n} ( x - x_{n-1})\|^2,
\end{equation*}
and
\begin{equation*} 
\|x - x_n\|^2 = \|x - x_0\|^2 \prod_{k=1}^n 
\left( 1 - \left\| P_{W_k} \left( \frac{ x-x_{k-1}}{\|x-x_{k-1}\|} \right) \right\|^2 \right).
\end{equation*}
\end{lemma}
\vspace{.1in}

The next theorem is proven in the same manner as Theorem 4.1 in \cite{CP}.

\begin{theorem} \label{moment-err-thm}
Let $\{W_n\}_{n=1}^{\infty}$ be independent random subspaces of $\mathbb{H}^d$.  Assume that each $W_n$ has the common Kaczmarz bound $0\leq\alpha_s\leq1$ of order $s$.  The error of the subspace action algorithm \eqref{ff-kacz} satisfies
\begin{equation}\label{moment-err-bnd}
\left( \E\|x-x_n\|^{2s} \right)^{1/s} \leq \alpha_s^{n}\|x-x_0\|^{2}.
\end{equation}
Moreover, if the common Kaczmarz bound is tight, then equality holds in \eqref{moment-err-bnd}.
\end{theorem}

The next result shows that the subspace action algorithm \eqref{ff-kacz} remains robust when it only has access to noisy measurements $y^*_n = P_{W_n}(x) + \epsilon_n$.
The proof is similar to the work in \cite{N10} for the randomized Kaczmarz algorithm, but for the sake of completeness we include a proof in the Appendix.
The assumption $\epsilon_n\in W_n$ is reasonable when one has precise knowledge of $W_n$, since the noisy measurements $y_n^*$ can be projected onto $W_n$ during reconstruction and hence can be replaced, without loss of generality, by $P_{W_n}(y_n^*)=P_{W_n}(x)+P_{W_n}(\epsilon_n)$.

\begin{theorem}\label{thm:noise}Let $\{W_n\}_{n=1}^{\infty}$ be independent random subspaces of $\mathbb{H}^d$.  Assume that each $W_n$ has the common Kaczmarz bound $0<\alpha_s<1$ of order $s>0$.  Let $x_n^*$ be the iterations generated by 
the subspace action method $x_n^* = x^*_{n-1} +y_n^* -P_{W_n}(x_{n-1}^*)$
using the noisy measurements $y_n^*=P_{W_n}(x)+\epsilon_n$ with $\epsilon_n\in W_n$ and $\|\epsilon_n\| \leq \epsilon$. Then
$$\forall \thinspace 0<s \leq 1, \ \ \ \E\|x_n^*-x\|^{2s}\leq \alpha_s^{ns}\|x^*_0-x\|^{2s}+ \left( \frac{1}{1-\alpha_s^s} \right) \epsilon^{2s},$$
and
$$\forall s\geq 1, \ \ \ \left( \E\|x_n^*-x\|^{2s} \right)^{1/s} \leq \alpha_s^{n}\|x^*_0-x\|^{2}+ \left( \frac{1}{1-\alpha_s} \right) \epsilon^{2}.$$
\end{theorem}
\begin{proof}
See Appendix.
\end{proof}

The next result follows from Theorem \ref{moment-err-thm}, Lemma \ref{log-lim-lemma} and 
Definition \ref{kacz-bnd-def}. The proof is similar to the proof of Theorem 6.3 in \cite{CP}, but we include
details in the Appendix for the sake of completeness.
\begin{theorem} \label{log-moment-err-thm}
Suppose that the random subspaces $\{ W_n\}_{n=1}^{\infty}$ of $\mathbb{H}^d$ are independent and identically distributed versions of a random subspace $W$ that has the logarithmic Kaczmarz bound 
$0\leq\alpha_{\log}\leq1$.  The error of the subspace action algorithm \eqref{ff-kacz} satisfies
\begin{equation}\label{log-moment-err-bnd}
{\rm exp} \left( \mathbb{E} [\log \|x - x_n\|^2] \right) \leq \alpha_{\log}^n \|x-x_0\|^{2}.
\end{equation}
\end{theorem}

\begin{proof}
See Appendix.
\end{proof}

While the error bounds \eqref{moment-err-bnd} and \eqref{log-moment-err-bnd} are natural, it is important to note that they are simply upper bounds and need not be sharp.
The following example shows that in certain cases the error moments can be much smaller than the estimates obtained using Kaczmarz bounds
and thereby illustrates some practical limitations of Theorem \ref{moment-err-thm}.

\begin{example}\label{exa:onb}
Let $\{e_k\}_{k=1}^d$ be an orthonormal basis for $\mathbb{H}^d$ and let $V_k = {\rm span} (e_k)$. Let $W$ be the random 1-dimensional subspace defined by 
\begin{equation*}
\forall 1\leq k\leq d, \ \ \  \Pr[W= V_k]=1/d.
\end{equation*}
Suppose that $\{W_n\}_{n=1}^N$ are i.i.d. versions of $W$.  

To begin, recall that when $0<s<1$, $W$ has the Kaczmarz bound $\alpha_s = (1-1/d)$, see Example \ref{onb-kacz-bnd-example}.  So, by Theorem \ref{moment-err-thm}
the error for subspace action algorithm \eqref{ff-kacz} satisfies
\begin{equation} \label{onb-examp-kacz-errbnd}
\mathbb{E}\|x - x_N\|^{2s} \leq (1 - 1/d)^{sN} \| x - x_0\|^{2s}.
\end{equation}

Next, a more detailed analysis will show that the error bound \eqref{onb-examp-kacz-errbnd} can be significantly improved in this example.
Let the random variable $K_j$ denote the number of $\{W_n\}_{n=1}^N$ which equal $V_j$.
Since the $\{e_n\}_{n=1}^d$ are orthonormal, the projections $P_{W_n^{\perp}}$ commute, namely $P_{W_j^{\perp}} P_{W_k^\perp} = P_{W_k^\perp} 
P_{W_j^\perp}$ for all $j,k$.  This together with \eqref{error-nonorm} gives
$$\|x - x_N\|^2 = \| P_{W_N^{\perp}} \cdots P_{W_1^{\perp}} (x_0 - x)\|^2 = \| P_{V_1^\perp}^{K_1}  P_{V_2^\perp}^{K_2} \cdots  P_{V_d^\perp}^{K_d} (x - x_0)\|^2.$$

Since the $\{e_j\}_{j=1}^d$ are orthonormal it can be verified that if $K_j\neq 0$ holds for all $1\leq j \leq d,$ then $P_{V_1^\perp}^{K_1}  P_{V_2^\perp}^{K_2} \cdots  P_{V_d^\perp}^{K_d}=0.$  
Let $A_j$ denote the event that $K_j=0$, and note that $\Pr[A_1\cap A_2 \cap \cdots \cap A_n] = (1 - n/d)^N$.  
Let $A = \bigcup_{j=1}^d A_j$, and let $\chi_A$ denote the indicator function of the event $A$.
Thus
\begin{equation} \label{err-eq-chiA}
\mathbb{E}\|x- x_N\|^{2s} = \mathbb{E} \left( \| P_{V_1^\perp}^{K_1}  P_{V_2^\perp}^{K_2} \cdots  P_{V_d^\perp}^{K_d} (x - x_0)\|^{2s} \chi_A \right).
\end{equation}

To obtain an upper bound on $\mathbb{E}\|x - x_N\|^{2s}$, note that the projections $P_{V_j^{\perp}}$ have norm one, and apply the inclusion-exclusion principle
to \eqref{err-eq-chiA} as follows
\begin{align}
\mathbb{E}\|x - x_N\|^{2s}  & \leq \|x-x_0\|^{2s} \mathbb{E} [ \chi_A]=\|x - x_0\|^{2s} {\Pr} \left( \bigcup_{j=1}^d A_j \right) \notag \\
& = \|x - x_0\|^{2s} \sum_{k=1}^d (-1)^{k+1}  \binom{d}{k} {\Pr} \left( \bigcap_{j=1}^k A_j \right)\notag \\
& = \|x - x_0\|^{2s} \sum_{k=1}^d (-1)^{k+1}  \binom{d}{k}  \left( 1 - \frac{k}{d}\right)^N. \label{gen-upperbnd-onb-examp}
\end{align}
Keeping only the $k=1$ term in the sum \eqref{gen-upperbnd-onb-examp} gives
\begin{equation} \label{onb-example-better-bnd}
\mathbb{E}\|x_N - x\|^{2s}  \leq d(1 - 1/d)^N \|x - x_0\|^{2s}.
\end{equation}
When $0<s<1$ (and especially when $s$ is near 0), the error bound \eqref{onb-example-better-bnd} is smaller than the error bound \eqref{onb-examp-kacz-errbnd} 
for adequately large $N$.

Finally, it is worth noting that the upper bound \eqref{gen-upperbnd-onb-examp} cannot be significantly improved.
To see this, consider the case when $x-x_0$ satisfies $| \langle (x-x_0), e_n \rangle| \geq C>0$ for all $1 \leq n \leq d$.  
Since the $\{ e_n\}_{n=1}^d$ are orthonormal, it can be shown that if $K_j=0$ for some $1\leq j \leq n$ then 
$$\| P_{V_1^\perp}^{K_1}  P_{V_2^\perp}^{K_2} \cdots  P_{V_d^\perp}^{K_d} (x - x_0)\| \geq C.$$
Thus \eqref{err-eq-chiA} and similar steps as for \eqref{gen-upperbnd-onb-examp} give
\begin{equation*} \label{onb-examp-lowerbnd}
\mathbb{E}\|x_n - x\|^{2s}  
\geq C^{2s} \mathbb{E} (\chi_A)= C^{2s} \ {\Pr} \left( \bigcup_{j=1}^d A_j \right) =  C^{2s} \sum_{k=1}^d (-1)^{k+1}  \binom{d}{k}  \left( 1 - \frac{k}{d}\right)^N.
\end{equation*}
\end{example}

Theorems \ref{moment-err-thm} and \ref{log-moment-err-thm} give upper bounds 
on error moments for the algorithm \eqref{ff-kacz} in terms of the Kaczmarz bounds $\alpha_s$ and $\alpha_{\log}$.
Kaczmarz bounds can similarly be used to control rates of almost sure convergence.
The next theorem can be proven in the same manner as either one of the two different proofs of Theorem 6.2 in \cite{CP}.
\begin{theorem}\label{pro_gen}
Let $\{W_k\}_{k=1}^{\infty}$ be independent random subspaces of $\mathbb{H}^d$.  Let $s>0$ be fixed and suppose that each $W_k$ has the common Kaczmarz bound $0<\alpha_s<1$ of order $s$. The error in the subspace action algorithm \eqref{ff-kacz}  satisfies
$$\forall\ 0<r< 1/\alpha_s, \quad \lim_{n\rightarrow\infty} r^n \|x-x_n\|^2=0, \ \ \hbox{ almost surely.}$$ 
\end{theorem}

The next theorem is proven in the same manner as Theorem 6.3 in \cite{CP}.
\begin{theorem}\label{thm_gen}
Suppose that the random subspaces $\{ W_n\}_{n=1}^{\infty}$ of $\mathbb{H}^d$ are independent and identically distributed versions of a random subspace $W$ that has the logarithmic Kaczmarz bound 
$0<\alpha_{\log}<1$.
The error in the subspace action algorithm \eqref{ff-kacz}  satisfies
$$\forall \ 0 < r < 1/\alpha_{\log}, \ \ \ \lim_{n\to \infty} r^n\|x-x_n\|^2 =0, \ \ \hbox{ almost surely.}$$
\end{theorem}

\begin{remark}[Randomized subspace actions for fusion frames] \label{remark-rand-alg-specialized}
The error bounds of this section apply to deterministic fusion frames through the randomized approach in Example \ref{example-randomized-ff}.
Let $\{(U_n, v_n)\}_{n=1}^N$ be a fusion frame for $\mathbb{H}^d$ with frame bounds $0<A\leq B$.
Let $W$ be the random subspace defined by $\Pr(W=U_n)=\frac{v_n^2}{\sum_{k=1}^N v_k^2}$
and let $\{W_n\}_{n=1}^{\infty}$ be i.i.d. versions of the random subspace $W$.
To recover $x\in \mathbb{H}^d$ from the projections $y_n = P_{U_n}(x), 1 \leq n \leq N$, proceed as in Example \ref{example-randomized-ff}
by iterating \eqref{ff-kacz} with the random subspaces $\{W_n\}_{n=1}^{\infty}$.
Theorems \ref{moment-err-thm} and \ref{pro_gen}, together with Example \ref{rand-det-exam} for $\alpha_1$, show that the iterates $x_n$ in the subspace action method \eqref{ff-kacz} satisfy
$$\E\|x-x_n\|^2 \leq \left(1 - \frac{A}{\sum_{n=1}^N v_n^2}\right)^{n}\|x-x_0\|^{2}$$
and
$$\forall\ 0<r< \left(1 - \frac{A}{\sum_{n=1}^N v_n^2}\right)^{-1}, \quad \lim_{n\rightarrow\infty} r^n \|x-x_n\|^2=0, \ \ \hbox{ almost surely.}$$
\end{remark}

\section{Minimal Kaczmarz bounds and optimal distributions} \label{opt-dist-sec}

Theorems \ref{moment-err-thm} and \ref{log-moment-err-thm} show that smaller Kaczmarz bounds $\alpha_s$ or $\alpha_{\log}$
yield smaller upper bounds on various error moments in \eqref{moment-err-bnd}
and \eqref{log-moment-err-bnd}.
In view of this, it is natural to ask which distributions on the probabilistic fusion frame  $W$ give the smallest Kaczmarz bound $\alpha_s$ or smallest logarithmic Kaczmarz bound $\alpha_{\log}$. 
For this question to be nontrivial, we shall restrict our attention to random subspaces $W\in G(k,d)$
with fixed dimension $1\leq k<d$.  Indeed, without fixing the dimension $k$, the trivial case 
$W = \mathbb{H}^{d}$ would ensure that $\alpha_s = \alpha_{\log} =0$ and that the algorithm \eqref{ff-kacz} converges exactly
after one step.

So, in this section we shall assume that $W\in G(k,d)$ is a random subspace, and we 
consider the question of which distributions on $W$ 
minimize
$$\sup_{x\in\mathbb{S}^{d-1}}  \left(  \mathbb{E} (1-\|P_{W}(x)\|^2)^s \right)^{1/s}
\ \ \ \hbox{ or } \ \ \ 
\sup_{x\in\mathbb{S}^{d-1}} \exp \left(  \mathbb{E} \left[\log(1-\|P_{W}(x)\|^2) \right] \right).$$
Equivalently, we seek to determine which Borel probability measures $\mu$ on $G(k,d)$
minimize each of the quantities
\begin{align} \label{kacz-s-prob}
F_{\alpha_s}(\mu) = \sup_{x\in\mathbb{S}^{d-1}}  \int_{G(k,d)} (1-\|P_{W}(x)\|^2)^s d \mu (W),
\end{align}
\begin{equation} \label{kacz-log-prob}
F_{\alpha_{\log}}(\mu) = \sup_{x\in\mathbb{S}^{d-1}}  \int_{G(k,d)} \log(1-\|P_{W}(x)\|^2)  d \mu(W).
\end{equation}
For related potential theoretic problems involving frames and fusion frames, see \cite{CF09, EO12, MRS}.
For related work on optimal randomizations in the Kaczmarz algorithm, see \cite{DSP14}.

It will be convenient to address \eqref{kacz-s-prob} and \eqref{kacz-log-prob} as special cases of the
more  general problem of finding 
which Borel probability measures $\mu$ on $X$ minimize
\begin{equation} \label{minprob}
\sup_{x\in X} \int_{Y} K(x,y) d\mu(y),
\end{equation}
where $X,Y$ are suitable homogeneous spaces and $K:X \times Y \to \R \cup \{- \infty\}$
is a suitable Borel measurable kernel.

We will start with a general analysis of the minimizer of \eqref{minprob}, which is of interest in its own right, and then 
specialize it to minimize \eqref{kacz-s-prob} and \eqref{kacz-log-prob} in Section \ref{inv-meas-min}.

\subsection{Kernels and potentials on homogeneous spaces} \label{homog-space-sec}

We shall assume throughout this section that $X$ and $Y$ are homogeneous spaces for a group $G$.
Specifically, we assume that $(X, d_X)$ and $(Y,d_Y)$ are compact metric spaces and that
$G$ is a compact topological group that acts isometrically and transitively on $X$ and $Y$. 
Recall that $G$ acts isometrically on $X$ if
$$\forall g \in G, \forall x_1,x_2 \in X,  \ \ \ d_X(g x_1, g x_2) = d_X(x_1,x_2),$$
and $G$ acts transitively on $X$ if
$$\forall x_1, x_2 \in X,  \exists g \in G, \ \ \hbox{ such that} \ \   x_1 =  g x_2.$$
In a given compact topological space $X$, we let $\mathcal{B}(X)$ denote the Borel $\sigma$-algebra and let $\mathcal{M}(X)$ denote the set of all
Borel probability measures on $(X,\mathcal{B}(X))$.

By the construction of Haar measure, see Part I, Section 1 in \cite{MS}, there exist unique Radon probability measures
$m_X \in \mathcal{M}(X)$ and $m_Y\in \mathcal{M}(Y)$ with the $G$-invariance property that for all 
$B_1 \in \mathcal{B}(X)$ and $B_2 \in \mathcal{B}(Y)$
\begin{equation} \label{inv-haar-meas}
\forall g \in G, 
\ \ \ m_X(g(B_1)) = m_X(B_1) \  \hbox{ and } \ m_Y(g(B_2)) = m_X(B_2),
\end{equation}
where $g(B_i) = \{ g(b) :  b \in B_i\}$.
The invariant measures $m_X$ and $m_Y$ are {\em strictly positive}, i.e., $m_X(\mathcal{O}_1)>0$ and $m_Y(\mathcal{O}_2)>0$ holds for all nonempty open sets $\mathcal{O}_1 \subset X$ and $\mathcal{O}_2 \subset Y$.  For compact spaces, this is a consequence of the transitivity of the group action.

In metric spaces, recall that a function $f:X \to \R\cup \{\pm \infty\}$ is {\em upper semi-continuous} if 
$\limsup_{n \to \infty} f(x_n) \leq f(x)$ whenever $\lim_{n \to \infty} x_n = x$ and $x \in X$, $\{x_n\}_{n=1}^{\infty} \subset X$.
In particular, if $f(x_0) < c$ for some $x_0 \in X, c \in \R$ then for every $\epsilon>0$ there exists an open neighborhood $\mathcal{O} \subset X$ containing $x_0$ such that
if $x\in \mathcal{O}$ then $f(x) < c + \epsilon$.

\begin{definition} \label{admiss-ker-def}
A Borel measurable function $K:X \times Y \to \R \cup \{- \infty\}$ will 
be said to be an {\em admissible kernel} if the following four conditions hold
\begin{equation}  \label{K-admiss-eq1}
\exists B_K \in [0,\infty), \forall x \in X, \forall y \in Y, \ \ \ 
-\infty \leq K(x,y) \leq B_K,
\end{equation}
and
\begin{equation}  \label{K-admiss-eq2}
\forall y \in Y, \ \ \  \int_{X} | K(x,y) | \thinspace dm_X(x) < \infty,
\end{equation}
and
\begin{equation}  \label{K-admiss-eq3}
\forall y \in Y, \hbox{ the function }  K(\thinspace \cdot \thinspace, y) \hbox{  is upper semi-continuous,}
\end{equation}
and
\begin{equation} \label{K-sym-O}
\forall g \in G, \forall x \in X,
\forall y \in Y, \ \ \ \ \ K(g (x), y) = K(x, g^{-1}(y)).
\end{equation}
\end{definition}
\vspace{.1in}

\begin{lemma}\label{lm:ck}
If $K$ is an admissible kernel then there exists a constant $C_K\in \R$ such that
\begin{equation} \label{int-const}
\forall \thinspace y \in Y, \ \ \ \int_{X} K(x,y)dm_X(x) = C_K.
\end{equation}
In view of \eqref{int-const}, we shall say that an admissible kernel $K$ has constant $C_K$.
\end{lemma}
\begin{proof}
The $G$-invariance of $m_X$ and \eqref{K-sym-O} imply that for any $g\in G$
and any $y \in Y$
\begin{align*}
\int_{X} K(x,g(y))dm_X(x)
=\int_{X} K(g^{-1}(x),y)dm_X(x)
=\int_{X} K(x,y)dm_X(x).
\end{align*}
Since the action of $g$ on $Y$ is transitive and by \eqref{K-admiss-eq2}
 this completes the proof.\\
\end{proof}

\begin{definition}
If $K$ is an admissible kernel and $\mu \in \mathcal{M}(Y)$ then the associated
{\em potential function} $U_K^{\mu}: X \to \R \cup \{\pm \infty\}$ is defined by
$$\forall x \in X, \ \ \ U^{\mu}_K(x) = \int_{Y} K(x,y) d\mu(y).$$
\end{definition}

\begin{lemma}\label{lem_const}
Let $K$ be an admissible kernel with constant $C_K$.
Given $\mu \in \mathcal{M}(Y)$ and $c\in \R$, {suppose that $U_K^{\mu}(x) =c$ 
for all $x\in X$.  Then $c=C_K$.}
\end{lemma}
\begin{proof}
Integrating both sides of \eqref{int-const} with respect to $d \mu$ and using \eqref{K-admiss-eq1} to apply the Fubini-Tonelli theorem gives that for any $x_0 \in \mathbb{S}^{d-1}$
\begin{align*}
C_K&=\int_{Y} \int_{X} K(x,y)dm_X(x)d\mu(y)
=\int_{X} \int _{Y} K(x,y)d\mu(y)dm_X(x)\\
&=\int_{X}  U^{\mu}_K(x)dm_X(x)
=\int_{X}  U^{\mu}_K(x_0)dm_X(x)
=U^{\mu}_K(x_0).
\end{align*}
\end{proof}

\begin{lemma}\label{rem:ck}
If $K$ is an admissible kernel with constant $C_K$, then $\forall x \in X, U_K^{m_Y}(x) =C_K.$
\end{lemma}
\begin{proof}
Computing similarly as in the proof of Lemma \ref{lm:ck} shows that the function $U_K^{m_Y}(x)=\int_{Y} K(x,y) dm_Y(y)$ is constant, and so by Lemma \ref{lem_const}, $U_K^{m_Y}(x)=C_K$ holds for all $x\in X$.
\end{proof}

\begin{lemma}
If $K$ is an admissible kernel and $\mu \in \mathcal{M}(Y)$
then the potential function $U^{\mu}_K$ is upper semi-continuous.
\end{lemma}

\begin{proof}
Let $x\in X$ be arbitrary and suppose that $\{x_n\}_{n=1}^{\infty} \subset X$ satisfies $\lim_{n\to \infty} x_n = x$.
Letting $B_K$ be as in \eqref{K-admiss-eq1}, it follows from \eqref{K-admiss-eq3} that
 for all $y \in Y$
 \begin{equation} \label{semicont-lemma-eq1}
B_K-K(x,y) \leq B_K- \limsup_{n\to \infty} K(x_n,y)  = \liminf_{n \to \infty} (B_K-K(x_n,y)).
\end{equation}
Applying Fatou's lemma to the nonnegative functions  $g_n(y)=B_K - K(x_n, y) \geq 0$ 
and using \eqref{semicont-lemma-eq1} gives
\begin{align*}
B_K -  U^{\mu}_K(x)
&= \int_{Y} (B_K - K(x,y)) d\mu(y) \\
&\leq  \int_{Y}  \liminf_{n \to \infty} (B_K -K(x_n,y)) d \mu(y)\\
& = \int_{Y} \liminf_{n \to \infty} g_n(y) d \mu(y) \\
& \leq \liminf_{n \to \infty} \int_{Y}  g_n(y) d \mu(y) \\
&= B_K - \limsup_{n \to \infty} \int_{Y}  K(x_n,y) d \mu(y) \\
&= B_K - \limsup_{n \to \infty}U^{\mu}_K(x_n).
\end{align*}
This shows that $U^{\mu}_K$ is upper semi-continuous.
\end{proof}

The following theorem characterizes minimizers $\mu \in \mathcal{M}(Y)$
of \eqref{minprob} in terms of the potential function $U^{\mu}_K$ being constant.
\begin{theorem}\label{thm_argmin}
Let $K$ be an admissible kernel with constant $C_K$
and let $\mu_0\in \mathcal{M}(Y)$.
The following are equivalent: \vspace{.05in}
\begin{enumerate}
\item[(1)] The  function $U^{\mu_0}_K$ is constant,\vspace{.05in}
\item [(2)] For every $x\in X$, there holds $U^{\mu_0}_K(x)=C_K$, \vspace{.05in}
\item[(3)]  $\mu_0$ minimizes \eqref{minprob}, that is, $\inf_{\mu \in \mathcal{M}(Y)}\sup_{x\in X}U^{\mu}_K(x)=
\sup_{x\in X}U^{\mu_0}_K(x).$
\end{enumerate}
\end{theorem}
\begin{proof} 

$(1) \iff (2)$.  Lemma \ref{lem_const} shows the equivalence of $(1)$ and $(2)$.\\

$(2) \Longrightarrow (3)$.  Proceed by contradiction and assume that there exists 
$\nu\in \mathcal{M}(Y)$ such that 
\begin{equation} \label{contra1}
\sup_{x\in X} U^{\mu_0}_K(x)>\sup_{x\in X} U^{\nu}_K(x).
\end{equation}
The definitions of $U_K^{\nu}$ and $C_K$, along with \eqref{contra1}, yield
\begin{align*}
C_K = \sup_{x\in X} U^{\mu_0}_K(x)>\sup_{x\in X} U^{\nu}_K(x) 
&\geq \int_{X} U_K^{\nu}(x) dm_X(x)\\
&=\int_{X} \int_{Y}  K(x,y)d\nu(y)dm_X(x)\\
&= \int_{Y} \int_{X} K(x,y) dm_X(x)d\nu(y)\\
& = \int_{Y}C_K  d\nu(y)= C_K.
\end{align*}
Thus, $C_K>C_K$ gives the desired contradiction.\\

$(3) \Longrightarrow (1)$. 
{Lemma \ref{rem:ck}} states that $U_K^{m_Y}(x) = C_K$ for all $x\in X$.
The previously proven implication $(2) \implies (3)$ shows that
\begin{equation} \label{gamma-minimizer-eq}
\inf_{\mu \in \mathcal{M}(Y)}\sup_{x\in X}U^{\mu}_K(x)=
\sup_{x\in X}U^{m_Y}_K(x)=C_K.
\end{equation}
The assumption (3) together with \eqref{gamma-minimizer-eq} gives
$$\sup_{x\in X}U^{\mu_0}_K(x)=
\inf_{\mu \in \mathcal{M}(Y)}\sup_{x\in X}U^{\mu}_K(x)=
\sup_{x\in X}U^{m_Y}_K(x)=C_K.$$

We proceed by contradiction and assume that $U^{\mu_0}_K$ is not constant, so that 
there exists $x_0\in X$ and $\epsilon>0$ with $U_K^{\mu_0}(x_0) <C_K - 2\epsilon$.
Since $U_K^{\mu_0}(x)$ is upper semi-continuous, there exists an open neighborhood $B\subset X$ containing $x_0$ such that 
$$\forall x \in B, \ \ \ U_K^{\mu_0}(x)<C_K-\epsilon.$$
Since $m_X$ is strictly positive, we have $m_X(B)>0$.
Thus,
 \begin{align}
\int_{X} U_K^{\mu_0}(x)dm_{X}(x)
&=\int_B U_K^{\mu_0}(x)dm_{X}(x)
+\int_{X\backslash B} U_K^{\mu_0}(x)dm_X(x)\notag\\
&\leq(C_K-\epsilon)m_X(B)
+C_Km_X(X \backslash B)\notag\\
&=C_K-\epsilon m_X(B)<C_K. \label{min-thm-contra-eq1}
\end{align}
On the other hand, the definitions of $U^{\mu_0}_K$ and $C_K$ along with the Fubini-Tonelli theorem give
\begin{align}
\int_{X} U_K^{\mu_0}(x)dm_X(x)
&=\int_{X} \int_{Y} K(x,y) d \mu_0(y) dm_X(x)\notag\\
&=\int_{Y} \int_{X} K(x,y)  dm_X(x)d \mu_0(y) \notag\\
&=\int_{Y} C_K d\mu_0(y)=C_K. \label{min-thm-contra-eq2}
\end{align}
Equations \eqref{min-thm-contra-eq1} and \eqref{min-thm-contra-eq2} yield the
desired contradiction $C_K<C_K$.
\end{proof}

Note that the assumption of upper semi-continuity was only needed for the proof of
the implication $(3) \implies (1)$.   Theorem \ref{thm_argmin}, together with Lemma \ref{rem:ck}, gives the following corollary.  

\begin{corollary} \label{inv-meas-min}
The invariant measure $m_Y \in \mathcal{M}(Y)$ satisfies
$$\inf_{\mu \in \mathcal{M}(Y)}\sup_{x\in X}U^{\mu}_K(x)=
\sup_{x\in X}U^{m_Y}_K(x).$$
In other words, the invariant measure $\mu = m_Y$ is a minimizer of \eqref{minprob}.
\end{corollary}

\subsection{Optimal distributions for subspace actions} \label{opt-dist-sec-grass}
We now specialize Theorem \ref{thm_argmin} to the problems \eqref{kacz-s-prob} and \eqref{kacz-log-prob} of
determining which distributions give minimal Kaczmarz bounds for the algorithm \eqref{ff-kacz}.  
We consider the case when $X=\mathbb{S}^{d-1}$, $Y=G(k,d)$, and $G$ is either 
the orthogonal group $\mathcal{O}(d)$ or unitary group $\mathcal{U}(d)$ depending on whether ${\mathbb{H}} = \R$ or $\mathbb{C}$.   

If $X=\mathbb{S}^{d-1}$ is endowed with the norm metric $d_X(u,v)=\|u - v\|$ from $\mathbb{H}^d$, then $G$ acts isometrically on $X$, and we 
let $\sigma_{d-1}= m_X$ denote the invariant measure as in \eqref{inv-haar-meas}.  
Next, endow $Y=G(d,k)$ with a metric $d_Y$ as follows.  Given subspaces $V,W\in G(k,d)$, let $P_V, P_W$ denote the matrix representations (with respect to the canonical basis) of the orthogonal projections onto $V$ and $W$ respectively, and define the metric $d_Y(V,W) =\|P_V-P_W\|_{\rm Frob}$ on $G(k,d)$. 
The group $G$ acts isometrically on $Y$ with this metric, and we 
let $\Gamma_{k,d} = m_Y \in \mathcal{M}(G(k,d))$ denote the invariant measure as in \eqref{inv-haar-meas}.
In particular, the invariant measure $\Gamma_{k,d}$ satisfies $$\forall S \in \mathcal{B}(G(k,d)), \ \ \ \forall g \in G \ \ \ \Gamma_{k,d}( g(S)) =  \Gamma_{k,d}(S),$$
where $g(S) = \{ g(W) : W \in S \}$.   

The following examples show that the kernels associated to problems \eqref{kacz-s-prob} and \eqref{kacz-log-prob} are admissible as in Definition \ref{admiss-ker-def}, and hence can be treated using the results of Section \ref{homog-space-sec}.

\begin{example}  \label{kernel-examp1}
Fix $s>0$ and define $K_{\alpha_s}:\mathbb{S}^{d-1} \times G(k,d) \to \R \cup \{- \infty\}$ by 
\begin{equation*}
K_{\alpha_s}(x,W) = \left( 1 - \|P_W(x)\|^2 \right)^s.
\end{equation*}
Then $K_{\alpha_s}$ is an admissible kernel.  Conditions \eqref{K-admiss-eq1}, \eqref{K-admiss-eq3}, and \eqref{K-sym-O} are easy to verify, and \eqref{K-admiss-eq2} holds since $0 \leq K_{\alpha_s}(x,W) \leq 1$ holds for all $x \in \mathbb{S}^{d-1}$ and $W \in G(k,d)$.
\end{example}

\begin{example}  \label{kernel-examp2}
Define $K_{\alpha_{\log}}:\mathbb{S}^{d-1} \times G(k,d) \to \R \cup \{- \infty\}$ by 
\begin{equation} \label{K-alpha-log-kermel-def}
K_{\alpha_{\log}}(x,W) = \log(1-\|P_W(x)\|^2).
\end{equation}
Then $K_{\alpha_{\log}}$ is an admissible kernel.
Conditions \eqref{K-admiss-eq3} and \eqref{K-sym-O}  are easy to verify, and \eqref{K-admiss-eq1} holds since $- \infty \leq K_{\alpha_{\log}}(x,W) \leq 0$
holds for all $x \in \mathbb{S}^{d-1}$ and $W \in G(k,d)$.  The next result, Lemma \ref{log-adm-integral-lemma}, verifies the condition \eqref{K-admiss-eq2} by direct computation.
\end{example}

\begin{lemma} \label{log-adm-integral-lemma}
Fix $1 \leq k<d$ and let $K_{\alpha_{\log}}$ be as in \eqref{K-alpha-log-kermel-def}.  Then $K_{\alpha_{\log}}$ satisfies \eqref{K-admiss-eq2}.  In other words,
\begin{equation} \label{log-adm-integral-lemma-eq}
\forall W \in G(k,d), \ \ \  \int_{\mathbb{S}^{d-1}} \left| \log( 1 - \|P_W(x)\|^2) \right| d \sigma_{d-1}(x)  < \infty.
\end{equation}
\end{lemma}
\begin{proof}
We begin with the real case $\mathbb{H}^d = \mathbb{R}^d$.  In this case, note that the measure $\sigma_{d-1}$ is defined by
\begin{equation} \label{normalized-hausdorff}
\forall E \in \mathcal{B}(\mathbb{S}^{d-1}), \ \ \ \sigma_{d-1}(E) = \frac{H_{d-1}(E)}{H_{d-1}(\mathbb{S}^{d-1})},
\end{equation}
where $H_{d-1}$ denotes the $(d-1)$-dimensional Hausdorff measure.

Fix $1 \leq k < d$.  Given $x= (x_1, \cdots x_d) \in \mathbb{H}^d$ denote $x^{\prime} = (x_1, \cdots, x_{k}) \in \mathbb{H}^{k}$ and $x^{\prime\prime} = (x_{k+1}, \cdots, x_d) \in \mathbb{H}^{d-k}$.  So that $x = (x^{\prime}, x^{\prime \prime})$ and
$\|x\|^2 = \|x^{\prime}\|^2 + \|x^{\prime \prime}\|^2$.
Using $\mathcal{O}(d)$-invariance of $\sigma_{d-1}$, we may assume without loss of generality that $W$ is the span of the first
$k$ canonical basis vectors, so that $P_W(x) = x^{\prime}$.  When $x \in \mathbb{S}^{d-1}$ we have $1 - \|P_W(x)\|^2 = 1 - \|x^{\prime}\|^2 = \|x^{\prime \prime}\|^2.$

It will be useful to recall the following change of variables formula in spherical coordinates, e.g., see Appendix D.2 in \cite{G},
$$\int_{R \mathbb{S}^{d-1}} f(x) dH_{d-1}(x) 
= \int_{-R}^R \left( \int_{\sqrt{R^2 - x_1^2} \ \mathbb{S}^{d-2}} f(x_1, x_2, \cdots, x_d) dH_{d-2}(x_2, \cdots, x_d) \right) \frac{R \ d x_1}{\sqrt{R^2 - x_1^2}}.$$
A repeated application of this leads to the following
\begin{equation} \label{sphere-coord-eq}
\int_{\mathbb{S}^{d-1}} f(x) dH_{d-1}(x) = 
\int_{\mathbb{B}_{d-k}} \left( \int_{\sqrt{1- \|x^{\prime \prime} \|^2} \thinspace \mathbb{S}^{k-1}} f(x^{\prime}, x^{\prime \prime}) \ dH_{k-1}(x^{\prime}) \right)  
\frac{dH_{d-k} (x^{\prime \prime})}{\sqrt{1- \|x^{\prime \prime} \|^2}},
\end{equation}
where $\mathbb{B}_{d-k} = \{ x^{\prime \prime} \in \mathbb{H}^{d-k} : \| x^{\prime \prime} \| \leq 1\}$ is the unit-ball in $\mathbb{H}^{d-k}$.

Using \eqref{normalized-hausdorff} and \eqref{sphere-coord-eq} yields
\begin{align}
H_{d-1}(\mathbb{S}^{d-1}) \int_{\mathbb{S}^{d-1}} & \log(1 - \|P_W(x) \|^2) d\sigma_{d-1}(x) = 2 \int_{\mathbb{S}^{d-1}} \log \|x^{\prime \prime} \| \ dH_{d-1}(x) \notag \\
&= 2\int_{\mathbb{B}_{d-k}} \left( \int_{ \sqrt{1 - \|x^{\prime \prime}\|^2} \thinspace \mathbb{S}^{k-1}} \log \|x^{\prime \prime}\| \ dH_{k-1}(x^{\prime}) \right)  
 \frac{dH_{d-k}(x^{\prime \prime})}{\sqrt{1- \|x^{\prime \prime} \|^2}} \notag \\
 &=2 \int_{\mathbb{B}_{d-k}} H_{k-1}(\sqrt{1- \|x^{\prime \prime} \|^2} \thinspace \mathbb{S}^{k-1}) \log \|x^{\prime \prime}\|
 \  \frac{dH_{d-k} (x^{\prime \prime})}{\sqrt{1- \|x^{\prime \prime} \|^2}} \notag \\
  &=2 H_{k-1}(\mathbb{S}^{k-1}) \int_{\mathbb{B}_{d-k}} \left(\sqrt{1- \|x^{\prime \prime} \|^2} \right)^{k-{1}} \log \|x^{\prime \prime}\|
 \ \frac{dH_{d-k} (x^{\prime \prime})}{\sqrt{1- \|x^{\prime \prime} \|^2}} \notag \\
  &= 2H_{k-1}(\mathbb{S}^{k-1}) \int_{\mathbb{B}_{d-k}} \left(\sqrt{1- \|x^{\prime \prime} \|^2} \right)^{k-{2}} \log \|x^{\prime \prime}\| \
  dH_{d-k} (x^{\prime \prime}) \notag \\
  &= 2H_{k-1}(\mathbb{S}^{k-1}) H_{d-k-1}(\mathbb{S}^{d-k-1}) \ \int_0^1 (1- \rho^2)^{\frac{k-2}{2}} \ \rho^{d-k-1} \ \log \rho \ d \rho. \label{log-int-adm-lem-eq}
    \end{align}
Since $\int_0^1\rho^a \log \rho \ d\rho$ is finite when $a>-1,$ it can be checked that the integral \eqref{log-int-adm-lem-eq} is finite when $1\leq k < d$.
Thus, \eqref{log-adm-integral-lemma-eq} holds when $\mathbb{H}^{d} = \R^d$.

The complex case $\mathbb{H}^d=\C^d$ follows from the real case by identifying $\mathbb{C}$ with $\mathbb{R}^2$.   
In particular, identify $\mathbb{S}^{d-1}_{\mathbb{C}}$ with $\mathbb{S}^{2d-1}_{\mathbb{R}}$, identify $G(k,\C^d)$ with $G(2k,\R^{2d})$,
identify Borel sets in $\C^d$ with  Borel sets in $\R^{2d}$,
and identify $\sigma_{d-1}(E)$ with $H_{2d-1}(E)/H_{2d-1}(\mathbb{S}^{2d-1})$.  With these identifications, in the complex case there holds
\begin{align*}
\int_{\mathbb{S}^{d-1}} & \log(1 - \|P_W(x) \|^2) d\sigma_{d-1}(x) \\
&= \frac{2H_{2k-1}(\mathbb{S}^{2k-1}) H_{2d-2k-1}(\mathbb{S}^{2d-2k-1})}{H_{2d-1}(\mathbb{S}^{2d-1})} \ \int_0^1 (1- \rho^2)^{\frac{2k-2}{2}} \ \rho^{2d-2k-1} \ \log \rho \ d \rho,
\end{align*}
which is finite when $1 \leq k <d$.
\end{proof}

The next result is a corollary of Theorem \ref{thm_argmin} and Corollary \ref{inv-meas-min} when $X=\mathbb{S}^{d-1}$, $Y=G(k,d)$.
Since a measure $\mu \in \mathcal{M}(G(k,d))$ may be associated to a random subspace $W$, 
we shall say that $\mu$ has Kaczmarz bounds $\alpha_s$ and $\alpha_{\log}$ if the associated random subspace $W$ has these Kaczmarz bounds.

\begin{corollary}\label{cor_argmin}
Let $\mu \in \mathcal{M}(G(k,d))$.
\begin{enumerate}
\item $\mu$ is a minimizer of \eqref{kacz-s-prob} if and only if
the Kaczmarz bound $\alpha_s$ is tight.
\item $\mu$ is a minimizer of \eqref{kacz-log-prob} if and only if
the logarithmic Kaczmarz bound $\alpha_{\log}$ is tight.
\item The invariant measure $\Gamma = \Gamma_{k,d} \in G(k,d)$ minimizes
both quantities \eqref{kacz-s-prob} and \eqref{kacz-log-prob}.
\end{enumerate}
\end{corollary}

Corollary \ref{cor_argmin} shows that random subspaces chosen according to the invariant distribution on $G(k,d)$ are minimizers for the problems \eqref{kacz-s-prob} and \eqref{kacz-log-prob}.  
The next example shows that 
the minimizer for \eqref{kacz-s-prob} is not necessarily unique; we shall further discuss uniqueness
for \eqref{kacz-log-prob} in Section \ref{unique-sec}.
\begin{example} \label{non-unique-example}
Let $\{u_n\}_{n=1}^{N} \subset \mathbb{H}^d$ be a unit-norm tight frame for $\mathbb{H}^d$ when $d \geq 2$.  Namely, suppose that each $u_n$ satisfies $\|u_n\|=1$
and that
$$\forall x \in \R^d, \ \ \ \sum_{n=1}^N | \langle x, u_n \rangle|^2 = \frac{N}{d} \|x\|^2.$$
For examples of unit-norm tight frames see \cite{CK10}.
Let $W\in G(1,d)$ be the random subspace defined by $P[W = {\rm span} (u_n) ] = 1/N$ for each $1 \leq n \leq N$.  Then,
$$\forall x\in \mathbb{S}^{d-1}, \ \ \ \mathbb{E} ( 1 - \|P_W(x)\|^2) = 1 - \frac{1}{d}>0.$$ 
In particular, $W$ has a tight Kaczmarz bound of order $s=1$, and by Corollary \ref{cor_argmin}, 
the measure associated to $W$ is a minimizer of \eqref{kacz-s-prob} but is not the invariant measure on $G(1,d)$.
\end{example}

\section{Uniqueness of optimal distributions in $G(1,d)$} \label{unique-sec}
This section addresses the uniqueness of the minimizers in \eqref{kacz-s-prob} and \eqref{kacz-log-prob} when the random subspaces $W$ have dimension $k=1$, i.e., $W \in G(1,d)$.  The main goal of this section is to show that the invariant distribution $\Gamma_{1,d}\in \mathcal{M}(G(1,d))$ is the {\em unique} minimizer for the Kaczmarz problem \eqref{kacz-s-prob} when $0<s<1$, and for the logarithmic Kaczmarz problem \eqref{kacz-log-prob}.  
For perspective, Example \ref{non-unique-example} shows that minimizers for the Kaczmarz problem of order $s=1$ in \eqref{kacz-s-prob} need not be unique.

The following theorem is the main result of this section.
\begin{theorem}\label{thm:unique}
Consider the problems of minimizing \eqref{kacz-s-prob} and \eqref{kacz-log-prob} for $G(1,d)$, i.e., when $k=1$.
\begin{enumerate}
\item If $k=1$ and $0<s<1$ then the invariant distribution $\Gamma_{1,d} \in \mathcal{M}(G(1,d))$ is the unique minimizer of \eqref{kacz-s-prob}.
\item  If $k=1$ then the invariant distribution $\Gamma_{1,d} \in \mathcal{M}(G(1,d))$ is the unique minimizer of \eqref{kacz-log-prob}.
\end{enumerate}
\end{theorem}
The proof of Theorem  \ref{thm:unique} will proceed by identifying each subspace $W \in G(1,d)$ with an element of $\mathbb{H}^{d^2}$, and then applying tools from potential theory.  In particular, we shall make use of the following classical potential theoretic lemma.  Note that this lemma involves {\em signed} measures.
Recall that $\int_E \int_E f(x,y) d\eta(x) d\eta(y)$ is absolutely integrable if $$\int_E \int_E |f(x,y)| d |\eta|(x) d|\eta|(y) < \infty.$$

\begin{lemma} \label{pot-theory-lemma}
Suppose that  ${E} \subset \mathbb{H}^d$ is compact and that $\eta$ is a signed Borel measure on $E$ with zero total mass $\int_{{E}} d \eta =0.$   
\begin{enumerate}
\item If $0<s<1$ and $\int_{{E}} \int_{{E}} \| x - y \|^{2s} d \eta(x) d\eta (y) \geq 0,$ then the signed measure $\eta$ is identically zero.
\item If $\int_{{E}} \int_{{E}} \log \| x - y \| d \eta(x) d\eta (y) =0$ and is absolutely integrable, then the signed measure $\eta$ is identically zero.
\end{enumerate}
\end{lemma}

For the real case $\mathbb{H}^d = \mathbb{R}^d$, part (1) of Lemma \ref{pot-theory-lemma} appears as Lemma 1 in \cite{B56}, 
and part (2) appears as Lemma $3^{\prime}$ in Section 33 of \cite{Frost35}, cf. \cite{Fug60b, Fug60}.
The complex case $\mathbb{H}^d = \mathbb{C}^d$ follows from the real case by identifying measures on $\mathbb{C}^d$ with measures on $\mathbb{R}^{2d}$.

We are now ready to prove Theorem \ref{thm:unique}.\\

\noindent {\em Proof of Theorem \ref{thm:unique}.}
{\em Step I.} We begin by identifying each subspace $W \in G(1,d)$ with a unit-norm element of $\mathbb{H}^{d^2}$.
Define the following set of $d\times d$ matrices
$$\mathcal{E}=\{M\in \mathbb{H}^{d\times d}: M=xx^* \hbox{ for some } x\in\Sb^{d-1}\}.$$
With slight abuse of notation, we shall interchangeably think of $\mathcal{E}$ as a subset $\mathbb{H}^{d^2}$ and also as a set of $d\times d$ matrices.
Since the rank one projections $xx^*$ have Frobenius norm one, the elements in $\mathcal{E}$ are unit-norm when viewed as elements of $\mathbb{H}^{d^2}$, i.e., $\mathcal{E} \subset \mathbb{S}^{d^2-1}$.  It can also be verified that $\mathcal{E}$ is a compact subset of $\mathbb{H}^{d^2}$.

Define the map $\mathcal{L}: G(1,d) \to \mathcal{E}$ by $\mathcal{L}(W)=xx^*$, where $x\in W\cap \Sb^{d-1}$ is any unit-norm element of $W$.  This map 
does not depend on the choice of $x$, and so is well-defined.  
The map $\mathcal{L}$ is an isometry when $G(1,d)$ is endowed with the 
metric $d_Y$ from the beginning of Section \ref{opt-dist-sec-grass} and $\mathcal{E}$ is endowed with norm metric on $\mathbb{H}^{d^2}$. 
The map $\mathcal{L}$ is also bijective since $xx^*=yy^*$ implies $x=cy$ for some $|c|=1$.

Given a Borel probability measure $\mu \in \mathcal{M}(G(1,d))$, the map $\mathcal{L}$ induces a Borel probability measure 
$\mathcal{L}_{\mu} \in \mathcal{M}(\mathcal{E})$ defined by
$$\forall E \in \mathcal{B}(\mathcal{E}), \ \ \ \mathcal{L}_{\mu} (E) = \mu( \mathcal{L}^{-1}(E)),$$
where $\mathcal{L}^{-1}(E) = \{ W \in G(1,d) : \mathcal{L}(W) \in E \}$ is in $\mathcal{B}(G(1,d))$ since $\mathcal{L}$ is an isometry.
Likewise, if $S\in \mathcal{B}(G(1,d))$ then $\mathcal{L}(S) = \{ \mathcal{L}(s) : s \in S \} \in \mathcal{B}(\mathcal{E})$.
The bijectivity of $\mathcal{L}$ guarantees a one-to-one correspondence between $\mu$ and $\mathcal{L}_{\mu}$ since
$$\forall S\in \mathcal{B}(G(1,d)), \ \ \ \mu(S)=\mathcal{L}_{\mu}(\mathcal{L}(S)).$$

{\em Step II.}  Next, define the map $\Omega: \mathbb{S}^{d-1} \to G(1,d)$ by $\Omega(u) = {\rm span}(u)$.  The map $\Omega$ is surjective, and $\Omega(u) = \Omega(v)$
if and only if $u = c v$ for some unimodular scalar $c \in \mathbb{H}$ with $|c|=1$.
Let $W \in G(1,d)$ and $x \in \mathbb{S}^{d-1}$ be arbitrary, and pick any $y \in \mathbb{S}^{d-1}$ such that $W = \Omega (y)$.
Note that 
\begin{align}
\| \mathcal{L}(W) - \mathcal{L}(\Omega(x))\|^2 & = \| yy^* - xx^*\|^2_{\rm Frob}\notag\\
&= {\rm Tr} \left( (yy^* - xx^*)^*(yy^* - xx^*) \right)\notag\\
&=2 - {\rm Tr} (yy^*xx^*) - {\rm Tr} (xx^*yy^*) \notag \\
&= 2 - 2| \langle x,y \rangle |^2\notag\\
&= 2 - 2 \|P_W(x)\|^2.\label{frob-eq}
\end{align}

\vspace{.1in}

{\em Step III.}  To prove part (1) of Theorem \ref{thm:unique}, recall that the invariant measure $\Gamma=\Gamma_{1,d}\in \mathcal{M}(G(1,d))$ is a 
minimizer of \eqref{kacz-s-prob} by Corollary \ref{inv-meas-min}.  Suppose that $\nu \in \mathcal{M}(G(1,d))$ is also a minimizer of \eqref{kacz-s-prob}.   
By Theorem \ref{thm_argmin}, there exists a constant $C \in \R$ such that
\begin{equation} \label{nu-gamma-C-eq}
\forall x \in \mathbb{S}^{d-1}, \ \ \ \int_{G(1,d)} (1 - \|P_W(x) \|^2)^s d \nu(W)  = \int_{G(1,d)} (1 - \|P_W(x) \|^2)^s d \Gamma (W) = C.
\end{equation}
This will imply
\begin{equation} \label{L-C-eq}
\forall Z \in \mathcal{E}, \ \ \  \int_{\mathcal{E}} \|Y -Z \|^{2s} d \mathcal{L}_\nu(Y) =  \int_{\mathcal{E}} \|Y -Z \|^{2s} d \mathcal{L}_\Gamma(Y)  =  2^s C.
\end{equation}

Equation \eqref{L-C-eq} is a result of \eqref{frob-eq}, \eqref{nu-gamma-C-eq}, and change of variables. Indeed, for any $Z\in\mathcal{E}$, there exists $x\in\Sb^{d-1}$ such that $\mathcal{L}(\Omega(x))=Z$, and 
\begin{align}\notag
\int_{\mathcal{E}} \|Y -Z \|^{2s} d \mathcal{L}_\nu(Y)
=& \int_{\mathcal{E}} \|Y -\mathcal{L}(\Omega(x)) \|^{2s} d \mathcal{L}_\nu(Y) \\\notag
=& \int_{\mathcal{L}^{-1}(\mathcal{E})}\|\mathcal{L}(W)-\mathcal{L}(\Omega(x))\|^{2s} d \nu(W)\notag \\\label{equ:sconst}
=& 2^s\int_{G(1,d)} (1 - \|P_W(x) \|^2)^s d \nu(W) =2^sC.
\end{align}
An identical computation for $\mathcal{L}_\Gamma$ yields \eqref{L-C-eq}.

Now define the signed measure $\eta = \mathcal{L}_{\nu} - \mathcal{L}_{\Gamma}$ on the Borel subsets of $\mathcal{E}$.  Since $\mathcal{L}_{\nu}$ and $\mathcal{L}_{\Gamma}$ are probability measures supported on $\mathcal{E}$, $\eta$ has zero total mass $\int_{\mathcal{E}} d \eta =0$.  Also, by \eqref{L-C-eq},
\begin{align}
\int_{\mathcal{E}} \int_{\mathcal{E}} \| Z - Y\|^{2s} d \eta(Y) d \eta(Z) & =
\int_{\mathcal{E}} \left(  \int_{\mathcal{E}} \| Z - Y\|^{2s} d \mathcal{L}_\nu(Y) -  \int_{\mathcal{E}} \| Z - Y\|^{2s} d \mathcal{L}_{\Gamma}(Y)\right)d \eta(Z) \notag \\
&=\int_{\mathcal{E}} \left( 2^sC - 2^sC \right) d \eta(Z)  = 0. \label{energy-eta-nonzero}
\end{align}
By Lemma \ref{pot-theory-lemma}, equation \eqref{energy-eta-nonzero} implies that the signed measure $\eta$ is identically zero.  Thus $\mathcal{L}_{\nu} = \mathcal{L}_{\Gamma}$.  Since the correspondence between a measure
$\mu \in \mathcal{M}(G(1,d))$ and the measure $\mathcal{L}_\mu \in  \mathcal{M}(\mathcal{E})$ is bijective, it follows that $\nu = \Gamma$.  
This establishes part (1) of Theorem \ref{thm:unique}.\\

{\em Step IV.}  The proof of part (2) of Theorem \ref{thm:unique} proceeds  similarly as Step III. 
Suppose that $\nu \in \mathcal{M}(G(1,d))$ is also a minimizer of \eqref{kacz-log-prob}.   
By Theorem \ref{thm_argmin}, there exists a constant $C_1 \in \R$ such that
\begin{equation}  \label{nu-gamma-C-eq2}
\forall x \in \mathbb{S}^{d-1}, \ \ \ \int_{G(1,d)} \log(1 - \|P_W(x) \|^2) d \nu(W)  = \int_{G(1,d)} \log(1 - \|P_W(x) \|^2) d \Gamma (W) = C_1.
\end{equation}
This will imply
\begin{equation} \label{L-C-eq2}
\forall Z \in \mathcal{E}, \ \ \  \int_{\mathcal{E}}\log \|Y -Z \| d \mathcal{L}_\nu(Y) = \int_{\mathcal{E}} \log \|Y -Z \| d \mathcal{L}_\Gamma(Y)  = 2^{-1}( C_1 + \log 2).
\end{equation}
This is a similar argument as \eqref{equ:sconst}. For any $Z\in\mathcal{E}$, we find $x$ such that $\mathcal{L}(\Omega(x))=Z$, then \eqref{frob-eq} and \eqref{nu-gamma-C-eq2} imply that
\begin{align}\notag
\int_{\mathcal{E}} \log\|Y -Z \| d \mathcal{L}_\nu(Y) 
=& \int_{\mathcal{E}} \log\|Y -\mathcal{L}(\Omega(x)) \| d \mathcal{L}_\nu(Y) \\\notag
=& \int_{\mathcal{L}^{-1}(\mathcal{E})}\log\|\mathcal{L}(W)-\mathcal{L}(\Omega(x))\|d \nu(W)\notag \\\label{equ:logconst}
=& 2^{-1}\int_{G(1,d)} \log(2 - 2\|P_W(x) \|^2) d \nu(W) =2^{-1}(\log2+C_1).
\end{align}
An identical computation for $\mathcal{L}_\Gamma$ yields \eqref{L-C-eq2}.

Now define the signed measure $\eta = \mathcal{L}_{\nu} - \mathcal{L}_{\Gamma}$ on the Borel subsets of $\mathcal{E}$.  
Since $\mathcal{L}_{\nu}$ and $\mathcal{L}_{\Gamma}$ are probability measures supported on $\mathcal{E}$, $\eta$ has zero total mass $\int_{\mathcal{E}} d \eta =0$.  Also, by \eqref{L-C-eq},
\begin{align}
\int_{\mathcal{E}} \int_{\mathcal{E}} \log \| Z - Y\| d \eta(Y) d \eta(Z) & =
\int_{\mathcal{E}} \left(  \int_{\mathcal{E}} \log\| Z - Y\| d \mathcal{L}_\nu(Y) -  \int_{\mathcal{E}} \log\| Z - Y\| d \mathcal{L}_{\Gamma}(Y)\right)d \eta(Z) \notag \\
&=\int_{\mathcal{E}} \left( 2^{-1}(C_1+\log2) - 2^{-1}(C_1+\log2) \right) d \eta(Z)  = 0. \label{energy-eta-nonzero2}
\end{align}

Note that $\int_{\mathcal{E}} \int_{\mathcal{E}} \log \| Z - Y\| d \eta(Y) d \eta(Z)$ is absolutely integrable.  To see this it suffices to show that
$\int_{\mathcal{E}}\left|\log\|Z-Y\|\right|d\mathcal{L}_{\nu}(Y)<\infty$ and $\int_{\mathcal{E}}\left|\log\|Z-Y\|\right|d\mathcal{L}_{\Gamma}(Y)<\infty$.
A similar computation as in \eqref{equ:logconst} shows that
$$\int_{\mathcal{E}}\left|\log\|Z-Y\|\right|d\mathcal{L}_{\nu}(Y)\leq 2^{-1}\left(\log2+\int_{G(1,d)} |\log(1 - \|P_W(x) \|^2)| d \nu(W)\right)=2^{-1}(\log2-C_1).$$
An identical computation also shows that $\int_{\mathcal{E}}\left|\log\|Z-Y\|\right|d\mathcal{L}_{\Gamma}(Y)<\infty.$

By Lemma \ref{pot-theory-lemma}, equation \eqref{energy-eta-nonzero2} implies that the signed measure $\eta$ is identically zero.  Thus $\mathcal{L}_{\nu} = \mathcal{L}_{\Gamma}$.  Since the correspondence between a measure
$\mu \in \mathcal{M}(G(1,d))$ and the measure $\mathcal{L}_\mu \in  \mathcal{M}(\mathcal{E})$ is bijective, it follows that $\nu = \Gamma$.  
This establishes part (2) of Theorem \ref{thm:unique}.\\
\qed

We conclude this section with some questions.
Theorem \ref{thm:unique} shows that the invariant measure $\Gamma_{k,d}$ is the unique minimizer of both \eqref{kacz-s-prob} with $0<s<1$ and \eqref{kacz-log-prob} 
for subspaces of dimension $k=1$, but it is not clear what happens for general values of $k$.
\begin{question}
Is the invariant measure $\Gamma_{k,d} \in \mathcal{M}(G(k,d))$ the unique minimizer of the problems \eqref{kacz-s-prob} with $0<s<1$ 
and \eqref{kacz-log-prob} for each $1 \leq k \leq d$?
\end{question}

Example \ref{non-unique-example} shows that when $s=1$ the minimizers of  \eqref{kacz-s-prob} are not unique.  
It would be interesting to understand the issue of uniqueness in \eqref{kacz-s-prob} when $s>1$.
\begin{question}
Is the invariant measure $\Gamma_{k,d} \in \mathcal{M}(G(k,d))$ the unique minimizer of the problem \eqref{kacz-s-prob} when $s>1$?
\end{question}

\section{Kaczmarz algorithm} \label{Kacz-section-at-end}
It is useful to briefly mention how the preceding results relate to the Kaczmarz algorithm \eqref{kacz-algorithm} for recovering $x\in \R^d$ from 
linear measurements $y_n = \langle x, \varphi_n \rangle, n \geq 1,$ 
when $\varphi_n \in \mathbb{S}^{d-1}$ are i.i.d. versions of a unit-norm random vector $\varphi \in \mathbb{S}^{d-1}$.

The random vector $\varphi \in \mathbb{S}^{d-1}$ induces a random subspace $W_\varphi \in G(1,d)$ by $W_\varphi={\rm span}(\varphi)$, 
and the Kacmarz bound $\alpha_s$ of order $s>0$ for $\varphi$ takes the following form, see Definition 2.3 in \cite{CP},
\begin{equation} \label{frame-kacz-bnd}
\alpha_s = \alpha_{s,\varphi}= \sup_{x \in \mathbb{S}^{d-1}} \left( \mathbb{E}[(1- | \langle x, \varphi \rangle|^2)^s] \right)^{1/s}
=\sup_{x \in \mathbb{S}^{d-1}} \left( \mathbb{E}[(1- \|P_{W_\varphi}(x)\|^2)^s] \right)^{1/s}.
\end{equation}
With slight abuse of notation, we say that the random vector $\varphi\in \mathbb{S}^{d-1}$ achieves minimal Kaczmarz bound for \eqref{frame-kacz-bnd}
if the probability measure $\mu_\varphi \in \mathcal{M}(G(1,d))$ associated to $W_\varphi$ is a minimizer of \eqref{kacz-s-prob}.
By Theorem 4.1 in \cite{CP}, the error for the Kaczmarz algorithm \eqref{kacz-algorithm} satisfies 
\begin{equation} \label{up-bound-in-kacz-section}
\left( \mathbb{E} \| x - x_n\|^{2s} \right)^{1/s} \leq \alpha_s^n \| x - x_0\|^{2}.
\end{equation}
Those $\varphi$ with minimal Kaczmarz bound give best control on the upper bound in \eqref{up-bound-in-kacz-section}.

For any fixed $s>0$, it follows from Corollary \ref{cor_argmin} that if $\varphi$ is uniformly distributed on $\mathbb{S}^{d-1},$ 
then $\varphi$ achieves the minimal Kaczmarz bound.
Moreover, when $0<s<1,$ Theorem \ref{thm:unique} shows that $\varphi$ achieves minimal Kaczmarz bound
only when $W_\varphi$ is distributed according to the invariant measure on $G(1,d)$.
For example, if $\psi$ is uniformly distributed on a hemisphere of $\mathbb{S}^{d-1},$ and $\varphi$ is uniformly distributed on $\mathbb{S}^{d-1}$,
then both $\varphi$ and $\psi$ achieve the minimal Kaczmarz bound since $W_\psi=W_\varphi$ are both distributed according to the invariant measure on $G(1,d)$.

\section{Numerical examples} \label{numerics-sec}

The numerical examples in this section plot error moments $\left( \mathbb{E}\|x - x_n\|^{2s} \right)^{1/s}$ of the subspace action 
algorithm \eqref{ff-kacz} for different random subspace distributions and different values of $s\geq 0$ (where
${\rm exp} (\mathbb{E} \log \|x - x_n\|^2)$ corresponds to $s=0$).
The expectations are approximated by averaging over 3000 trials in the first three examples, and by averaging over 9000 trials in the final example.  
All examples are done in real space $\R^d$, and the algorithm \eqref{ff-kacz} is implemented with the initial estimate $x_0=0$.

\begin{example}  \label{numexample-roots-unif}
Figure \ref{fig:dis} compares error moments for 
three different types of one-dimensional random subspaces in $\R^2$.  Define the $K$th roots of unity
frame $\Phi_K = \{\varphi^K_k\}_{k=1}^K \subset \R^2$ by
\begin{equation*}
\forall 1 \leq k \leq K, \ \ \ \varphi^K_k = ( \cos(2 \pi k/K) , \sin(2 \pi k/K) ).
\end{equation*}
It is well known that $\Phi_K$ is a unit-norm tight frame for $\R^2$ when $K \geq 3$, e.g., \cite{BPY}.  

When $K=3$ and $K=5$ we consider the random subspace distributions that are defined as in Example \ref{non-unique-example} 
by randomly selecting at uniform the span of an element in $\Phi_K$.  We also consider the invariant distribution $\Gamma_{1,2}$.  

Note that the Kaczmarz bounds of $\Gamma_{1,2}$ can  be explicitly computed when $s=2,1,1/2,0$.  
When $s=0$, Example 5.6 in \cite{CP} shows that $\alpha_{\log} = 1/4$.  For $s>0$ and $u\in \mathbb{S}^{d-1}$,
$$\left( \mathbb{E} (1 - \|P_W(u)\|^2)^s \right)^{1/s} = \left( \frac{1}{2\pi} \int_0^{2\pi} |\sin \theta|^{2s}  \ d \theta \right)^{1/s},$$
so that $\alpha_2 =\sqrt{3/8}$, $\alpha_1=1/2$, and $\alpha_{1/2}=(2/\pi)^2$.

The subplots in Figure \ref{fig:dis} show error moments for the different values of $s=2,1,1/2,0$ when $x=(0.2296, 0.9361)$.
Each subplot plots error moments
for each of the three random subspace distributions considered (randomized 3rd roots of unity, randomized 5th roots of unity, and the invariant
distribution), and for comparison also plots the theoretical upper bounds $\alpha_s^n$ from 
Theorem \ref{moment-err-thm} and Theorem \ref{log-moment-err-thm}.
Note that the uniform distribution has the smallest error among these three distribution and roughly coincides with $\alpha_s^n$.
In the case $s=1$, all plots are very close to each other, which agrees with the fact
that all three distributions have the same tight Kaczmarz bound $\alpha_1 = 1/2$.
\end{example}

\begin{example}\label{exa:fds}
Figure \ref{fig:fdis}  compares  error moments for three different types of two-dimensional random subspaces in $\R^5$. 
Draw 5 two-dimensional subspaces $W^5_n \subset \R^5, 1 \leq n \leq 5,$ independently at random according to the uniform distribution on $G(2,5)$.  
The resulting subspaces are now a deterministic collection.
Let $U$ be the random subspace defined by
$$\forall \thinspace 1 \leq n \leq 5, \ \ \ \Pr [ U = W^5_n ] = 1/5.$$
This will be referred to as the randomized 5 subspace distribution in this example. 
Similarly, draw 8 two-dimensional subspaces $W^8_n \subset \R^5, 1 \leq n \leq 8,$ independently at random according to the uniform distribution on $G(2,5)$,
and let $V$ be the random subspace defined by
$$\forall \thinspace 1 \leq n \leq 8, \ \ \ \Pr [ V = W^8_n ] = 1/8.$$
This will be referred to as the randomized 8 subspace distribution in this example.  See \cite{B13} for results concerning approximate tightness of randomly drawn fusion frames.

The subplots in Figure \ref{fig:fdis} show error moments for the different values of $s=2,1,1/2,0$ for the randomized 5 subspace distribution, the randomized 8 subspace distribution, and the invariant distribution $\Gamma_{2,5}$.  The signal $x$ was taken to be $x=(1,1,1,1,1)$.  In this experiment, the 8 subspace distribution outperforms the 5 subspace distribution, but the invariant distribution has the smallest error for each of the four values of $s$.
\end{example}

\begin{example}  \label{numexample-RONB}
Figure \ref{fig:ONBU} illustrates Example \ref{exa:onb} in $\R^{100}$, by comparing the invariant distribution $\Gamma_{1,100}$
with the randomized orthonormal basis distribution (RONB) defined by \eqref{onb-kacz-bnd-example-defeq}.
The figure compares error moments when $s=2,1,1/2$.  The signal $x$ was taken to be $x=(1,1,1,\cdots,1)$.
As expected from Example \ref{exa:onb}, the RONB distribution outperforms the invariant distribution when $s=1/2$ even though the invariant
distribution has a smaller Kaczmarz bound of order $1/2$.
When $s=1$ both distributions appear to give roughly the same error; this is consistent with the fact that both distributions have tight Kaczmarz bounds when $s=1$, cf. Corollary \ref{cor_argmin} and Example \ref{non-unique-example}.
Finally, note that the invariant distribution outperforms the RONB distribution when $s=2$.
\end{example}

\begin{example}\label{exa:fonb}
Figure \ref{fig:fONBU} illustrates an example that is similar to Example \ref{numexample-RONB}, but with four-dimensional subspaces in $\R^{100}$. 
Let $\{e_n\}_{n=1}^{100}$ be the canonical basis for $\R^{100}$, and for $1 \leq i \leq 25$, let $W_i = {\rm span} \{e_n\}_{n=4i-3}^{4i}$.
Note that $\{W_i\}_{i=1}^{25}$ is a tight fusion frame for $\R^{100}$ (with weights $v_i=1$).  Let $W$ be the random  four-dimensional subspace
defined by $\Pr[W = W_i] = 1/25$, and refer to this as the ONB distribution.  
The signal $x$ was taken to be $x=(1,1,1,\cdots,1)$.
Figure \ref{fig:fONBU} compares the ONB distribution with the invariant distribution $\Gamma_{4,100}$.
Similar to Example \ref{numexample-RONB}, the ONB distribution has smaller error moments than the invariant distribution when $0<s<1$ (even though the
invariant distribution uniquely achieves minimal Kaczmarz bounds when $0<s<1$).
\end{example}

\section*{Appendix}
\begin{proof}[Proof of Theorem \ref{thm:noise}] 
Since the random subspaces $\{W_j\}_{j=1}^n$ are independent, it follows that $W_n$ is independent of the random vector $x_{n-1}^*$.
Let $y_n=P_{W_n}(x)$ and $z_n = x_{n-1}^*+y_n-P_{W_n}(x_{n-1}^*)$.
It follows from Theorem \ref{moment-err-thm} (thinking of $x_1=z_n$ and $x_0=x^*_{n-1}$) that for all $s>0$
\begin{equation}\label{equ:normal}
\left( \E_{W_n}\|z_n-x\|^{2s} \right)^{1/s} \leq \alpha_s\|x_{n-1}^*-x\|^{2},
\end{equation}
where the expectation in \eqref{equ:normal} is with respect to $W_n$ only.

From the definition of $x^*_n$ we have $x_n^*=x_{n-1}^*+y_n+\epsilon_n-P_{W_n}(x_{n-1}^*)=z_n+\epsilon_n$, and hence
$x^*_n - x=z_n - x +\epsilon_n$.  Since $z_n-x=P_{W_n^{\perp}}(x_{n-1}^*-x)\in W_n^{\perp}$ and $\epsilon_n\in W_n$, it follows that 
\begin{equation} \label{pythag-eq}
\|x_n^*-x\|^2=\|z_n-x\|^2+\|\epsilon_n\|^2.
\end{equation}

{\em Case 1.} If $0<s\leq 1$ then \eqref{pythag-eq} gives
\begin{equation}\label{equ:s}
\|x_n^*-x\|^{2s}\leq\|z_n-x\|^{2s}+\|\epsilon_n\|^{2s}.
\end{equation} 
Since $\{W_j\}_{j=1}^n$ are independent, the expectation $\mathbb{E}\|x_n^*-x\|^{2s}$ can be written
as an iterated expectation $\mathbb{E}_{W_1, \cdots, W_{n-1}} \mathbb{E}_{W_n} \|x_n^*-x\|^{2s}$.
Combining  \eqref{equ:normal} and \eqref{equ:s} gives that 
\begin{align}
\E \|x_n^*-x\|^{2s} 
&= \E_{W_1, \cdots, W_{n-1}} \mathbb{E}_{W_n} \|x_n^*-x\|^{2s} \notag \\
& \leq \E_{W_1, \cdots, W_{n-1}}  \mathbb{E}_{W_n} \left( \|z_n-x\|^{2s} + \|\epsilon_n\|^{2s} \right)  \notag \\
& \leq \E_{W_1, \cdots, W_{n-1}} \left( \alpha_s^s \|x_{n-1}^*-x\|^{2s} + \epsilon^{2s}\right)  \notag \\
& = \alpha_s^s \thinspace \mathbb{E} \|x_{n-1}^*-x\|^{2s}+\epsilon^{2s}. \label{case1-iter-eq}
\end{align}
Iterating \eqref{case1-iter-eq} yields
\begin{align*}
\E\|x_n^*-x\|^{2s}&\leq \alpha_s^{ns}\|x^*_0-x\|^{2s}+\epsilon^{2s}\sum_{j=0}^{n-1}\alpha_s^{sj}
\leq \alpha_s^{ns}\|x^*_0-x\|^{2s}+\frac{1}{1-\alpha_s^s}\epsilon^{2s}.
\end{align*}

{\em Case 2.} If $s \geq 1$ then by \eqref{equ:normal}, \eqref{pythag-eq}, and Minkowski's inequality,
\begin{align}
(\E \|x_n^*-x\|^{2s} )^{1/s} 
&= \left( \E[ (\|z_n-x\|^{2}+\|\epsilon_n\|^{2})^s ]\right)^{1/s} \notag \\
&\leq \left( \E \|z_n-x\|^{2s} \right)^{1/s} + \left( \E \|\epsilon_n\|^{2s}\right)^{1/s} \notag \\
&= \left( \E_{W_1, \cdots, W_{n-1}} \mathbb{E}_{W_n} \|z_n-x\|^{2s} \right)^{1/s} + \epsilon^2 \notag \\
& \leq \left( \E_{W_1, \cdots, W_{n-1}}(\alpha_s^s \|x^*_{n-1}  - x\|^{2s}) \right)^{1/s} + \epsilon^2 \notag \\
&= \alpha_s \left( \E\|x^*_{n-1}  - x\|^{2s}\right)^{1/s} + \epsilon^2. \label{case2-iter-eq}
\end{align}
Iterating \eqref{case2-iter-eq} yields
$$(\E \|x_n^*-x\|^{2s} )^{1/s} \leq \alpha_s^n \|x_{0}^*-x\|^{2}  + \frac{1}{1-\alpha_s} \epsilon^2.$$

\end{proof}

\begin{proof}[Proof of Theorem \ref{log-moment-err-thm}]
Let $\epsilon>0$ be arbitrary.  Using Lemma \ref{log-lim-lemma}, for every $x\in \mathbb{S}^{d-1}$
there exists $s_x>0$ such that 
\begin{equation*}
\left( \mathbb{E} ( 1 - \| P_W(x)\|^2)^{s_x} ) \right)^{1/s_x} \leq \alpha_{\log}+\epsilon.
\end{equation*}
By the Lebesgue dominated convergence theorem, for every $x \in \mathbb{S}^{d-1}$
\begin{equation*}
\lim_{\|y\|=1, y \to x} \left( \mathbb{E} ( 1 - \| P_W(y)\|^2)^{s_x} ) \right)^{1/s_x}
= \left( \mathbb{E} ( 1 - \| P_W(x)\|^2)^{s_x} ) \right)^{1/s_x} \leq \alpha_{\log} + \epsilon.
\end{equation*}
Thus for every $x\in \mathbb{S}^{d-1}$ there exists an open neighborhood $U_x \subset \mathbb{S}^{d-1}$ about $x$ such that
\begin{equation*}
\forall y \in U_x, \ \ \ 
\left( \mathbb{E} ( 1 - \| P_W(y)\|^2)^{s_x} ) \right)^{1/s_x} \leq \alpha_{\log} + 2\epsilon.
\end{equation*}
By compactness there exists a finite set $\{x_n\}_{n=1}^N \subset \mathbb{S}^{d-1}$ such that $\cup_{n=1}^N U_{x_n}$ covers $\mathbb{S}^{d-1}$.  Letting
$s^*_{\epsilon} = \min\{s_{x_n}\}_{n=1}^N >0$ and using \eqref{lyap-ineq} gives
\begin{equation*}
\forall x \in \mathbb{S}^{d-1}, \ \ \ 
\left( \mathbb{E} ( 1 - \| P_W(x)\|^2)^{s^*_\epsilon} ) \right)^{1/s^*_{\epsilon}} \leq \alpha_{\log} + 2\epsilon.
\end{equation*}
Thus by Lemma \ref{log-lim-lemma} and Theorem \ref{moment-err-thm}
\begin{equation} \label{plus-eps-bnd}
{\rm exp} \left( \mathbb{E}\log \|x-x_n\|^2 \right) \leq 
\left( \mathbb{E}\|x-x_n\|^{2s^*_{\epsilon}} \right)^{1/s^*_{\epsilon}} 
\leq \left( \alpha_{\log} + 2\epsilon \right)^{n} \|x - x_0\|^2.
\end{equation}
Since $\epsilon>0$ is arbitrary, \eqref{plus-eps-bnd} yields \eqref{log-moment-err-bnd} as required.
\end{proof}

\section*{Acknowledgements}

The authors thank Jameson Cahill for valuable suggestions in the proof of Theorem \ref{thm:unique}.
The authors also thank Pete Casazza, Doug Hardin, Jiayi Jiang, Michael Northington and Ed Saff for helpful conversations
related to the material.  The authors were supported in part by 
NSF DMS Grant 1211687.  A.~Powell gratefully acknowledges the hospitality and support
of the Academia Sinica Institute of Mathematics (Taipei, Taiwan) and the City University of Hong Kong.

\begin{figure}[b]
        \centering
        \begin{subfigure}[b]{0.5\textwidth}
                \centering
                \includegraphics[width=\textwidth]{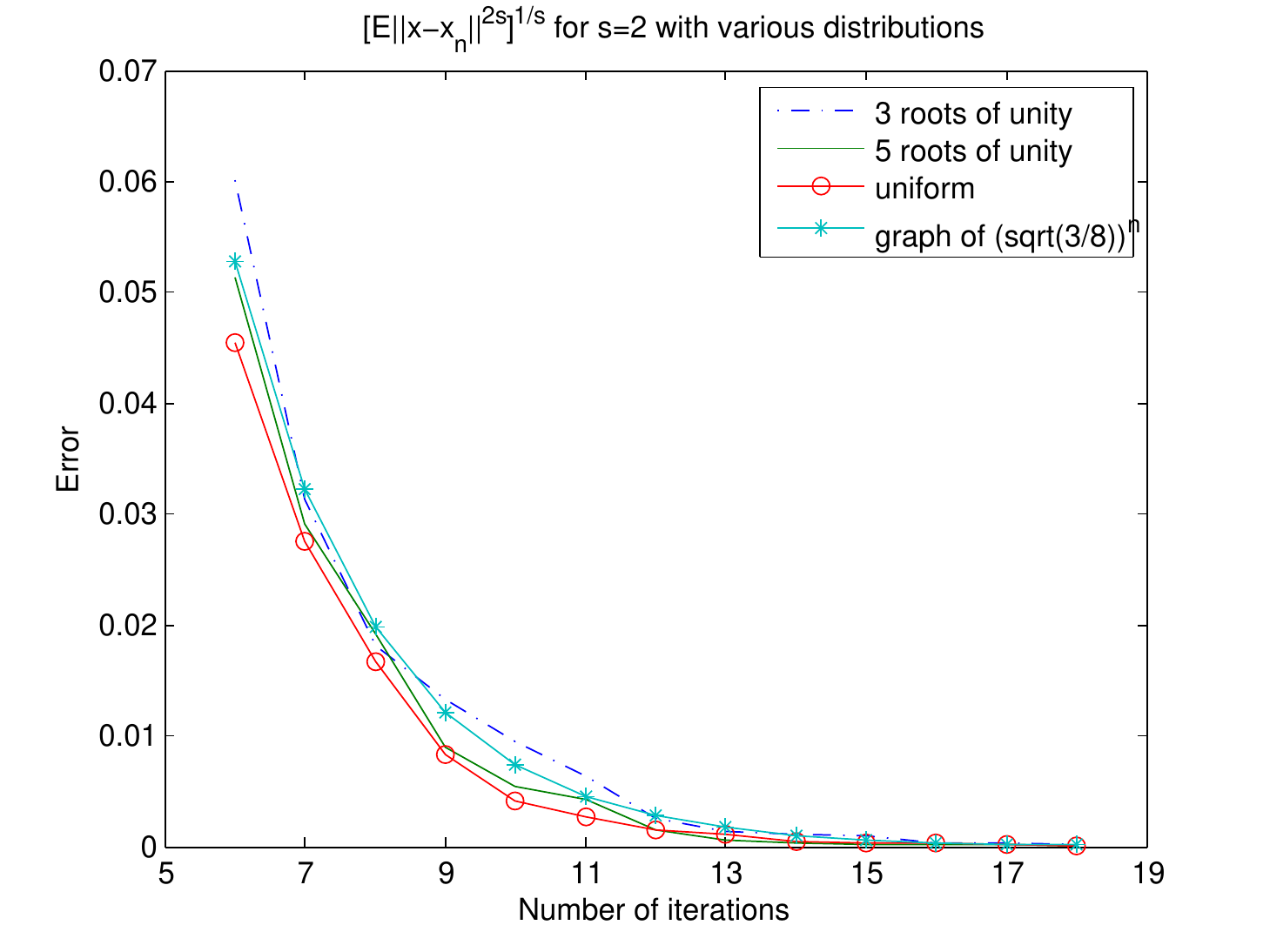}
               
                \label{fig:s2}
        \end{subfigure}
        ~\begin{subfigure}[b]{0.5\textwidth}
                \centering
                \includegraphics[width=\textwidth]{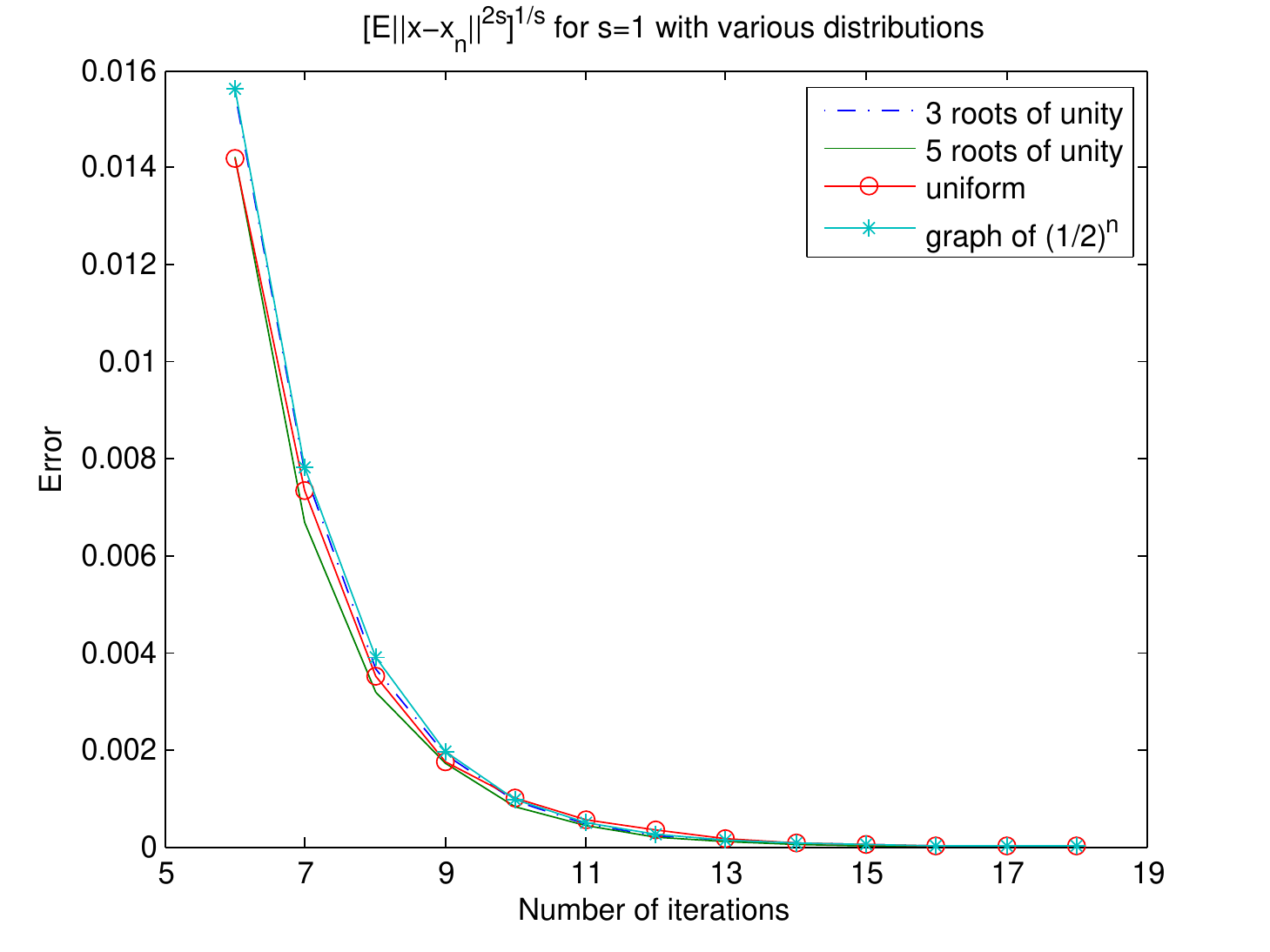}
                
                \label{fig:s1}
        \end{subfigure}
       
        \begin{subfigure}[b]{0.5\textwidth}
                \centering                \includegraphics[width=\textwidth]{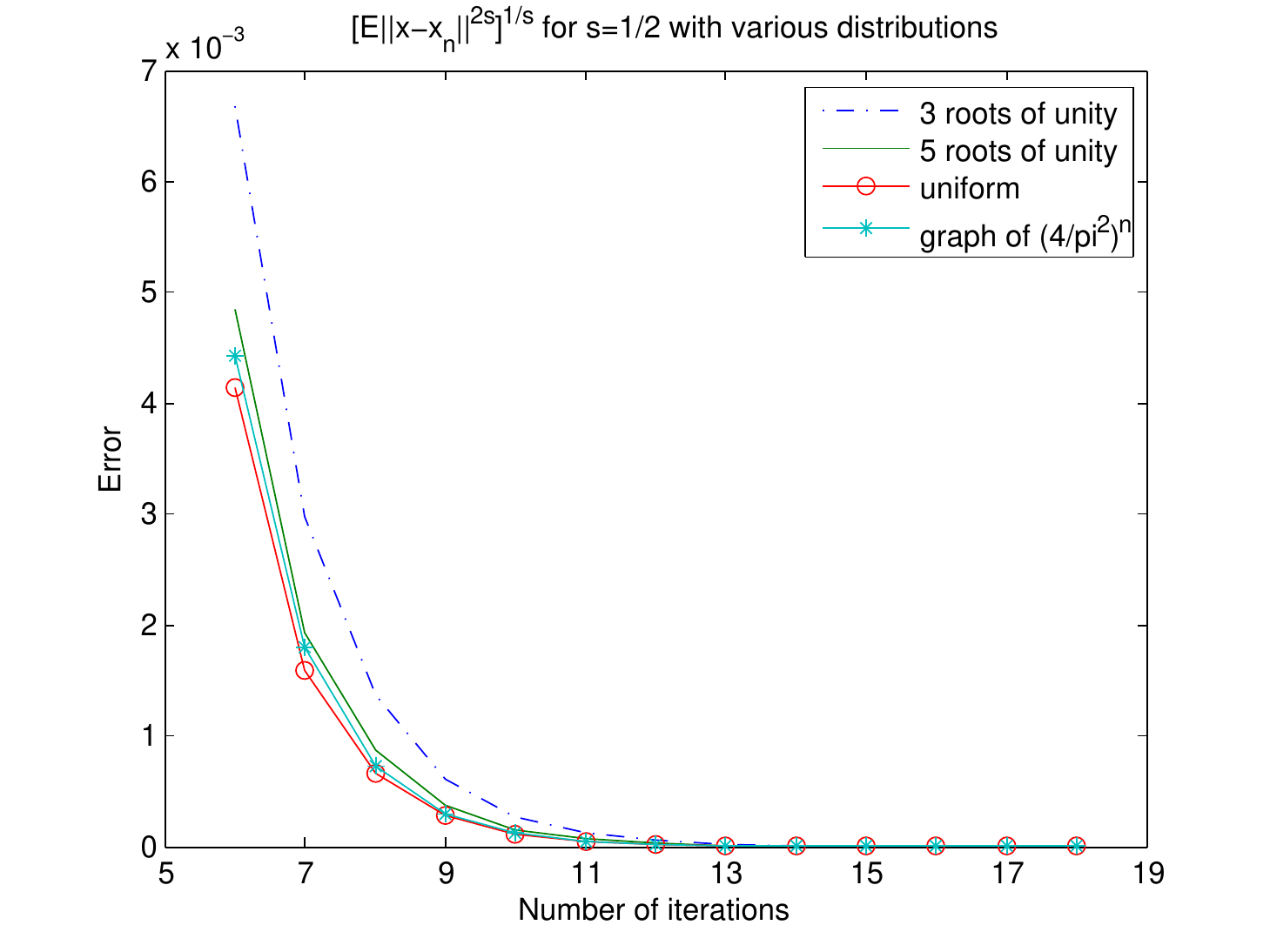}
                
                \label{fig:s05}
        \end{subfigure}
        ~\begin{subfigure}[b]{0.5\textwidth}
                \centering
                \includegraphics[width=\textwidth]{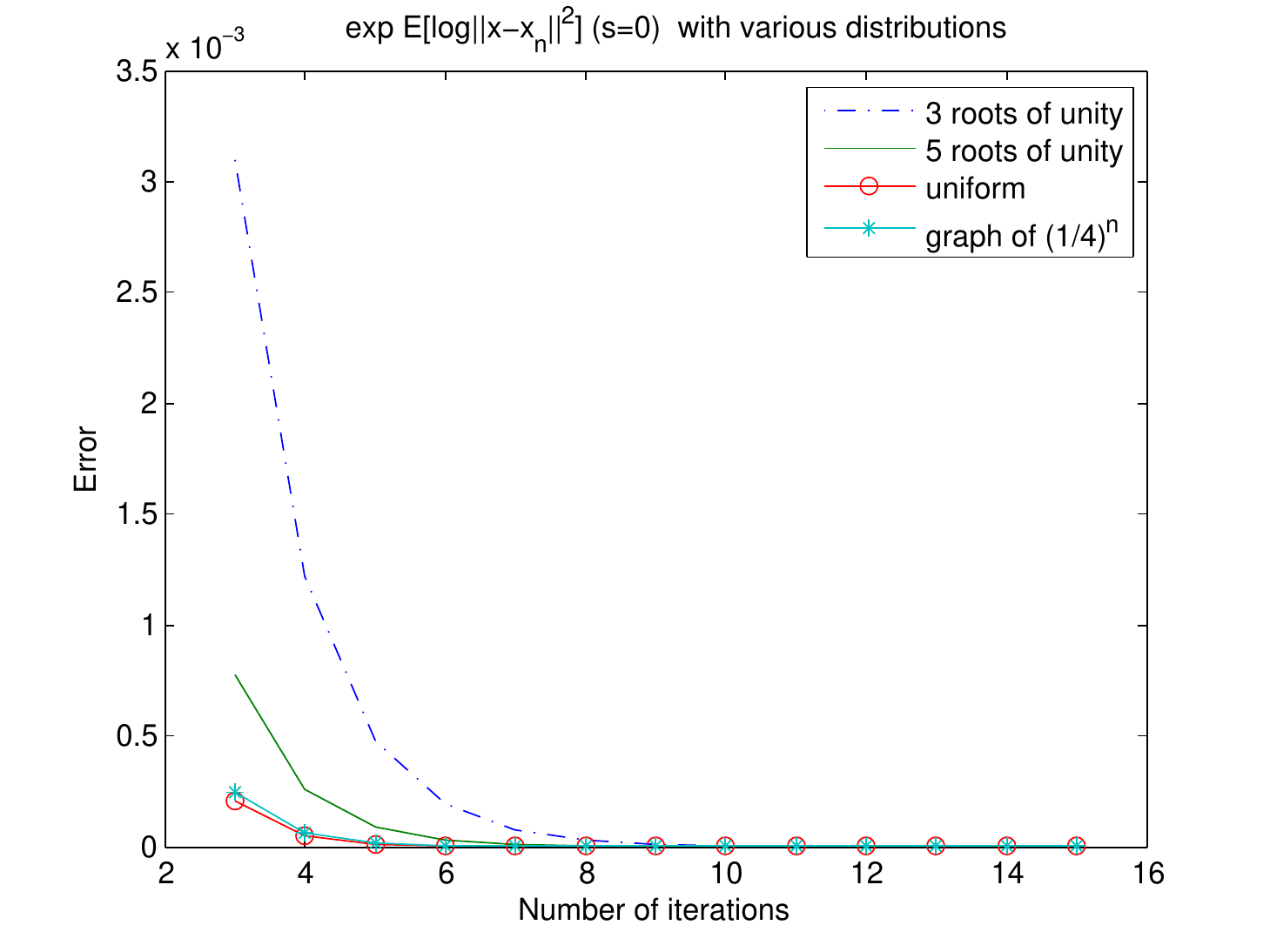}
                
                \label{fig:s0}
        \end{subfigure}
        \caption{Error moments for the invariant distribution $\Gamma_{1,2}$, the randomized 3rd roots of unity distribution, and the randomized 5th roots of unity distribution, along with the Kaczmarz upper bounds for $\Gamma_{1,2}$, see Example \ref{numexample-roots-unif}.}\label{fig:dis}
\end{figure}

\begin{figure}[b]
        \centering
        \begin{subfigure}[b]{0.5\textwidth}
                \centering
                \includegraphics[width=\textwidth]{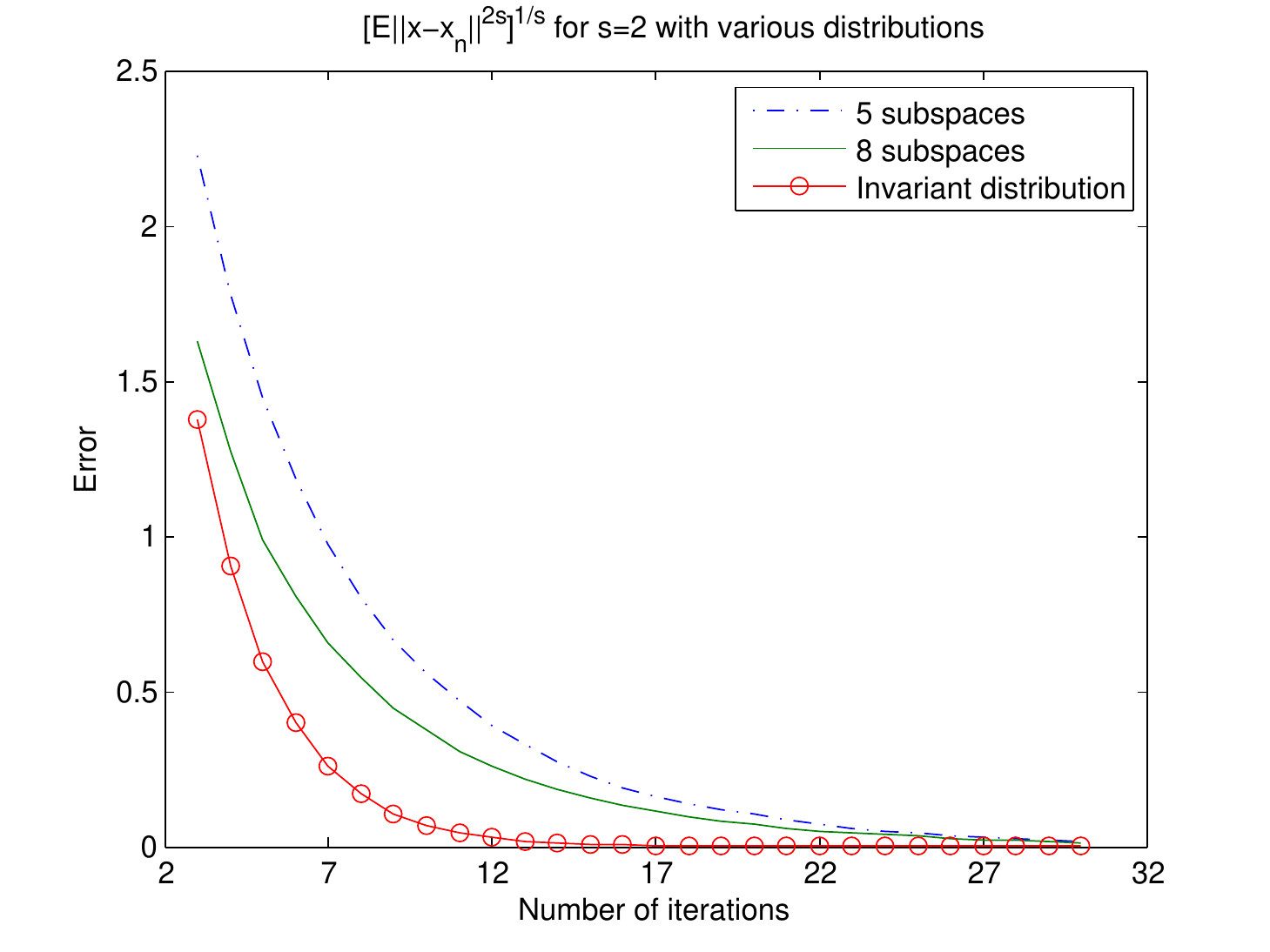}
               
                \label{fig:s2}
        \end{subfigure}
        ~\begin{subfigure}[b]{0.5\textwidth}
                \centering
                \includegraphics[width=\textwidth]{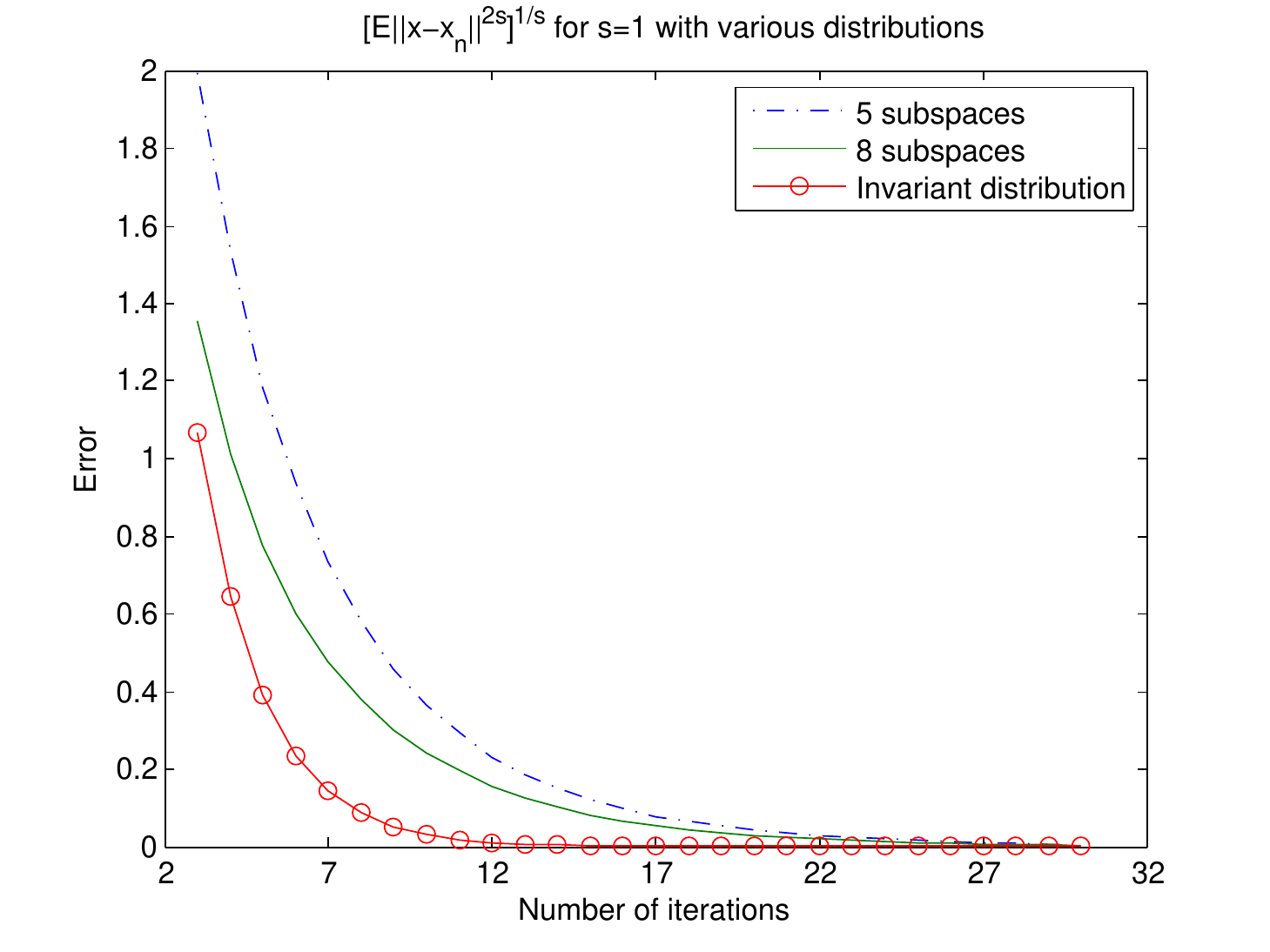}
                
                \label{fig:s1}
        \end{subfigure}
       
        \begin{subfigure}[b]{0.5\textwidth}
                \centering                \includegraphics[width=\textwidth]{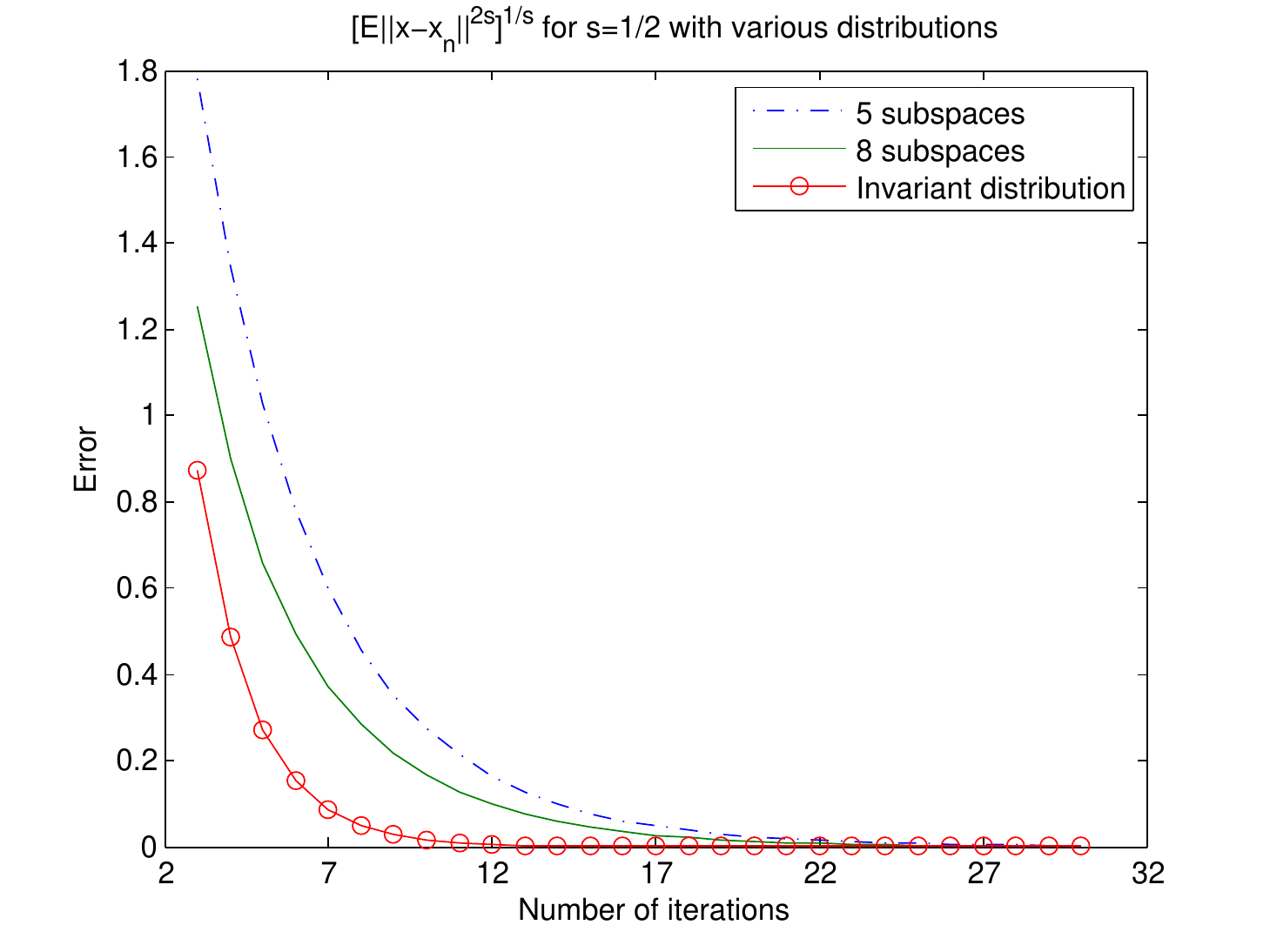}
                
                \label{fig:s05}
        \end{subfigure}
        ~\begin{subfigure}[b]{0.5\textwidth}
                \centering
                \includegraphics[width=\textwidth]{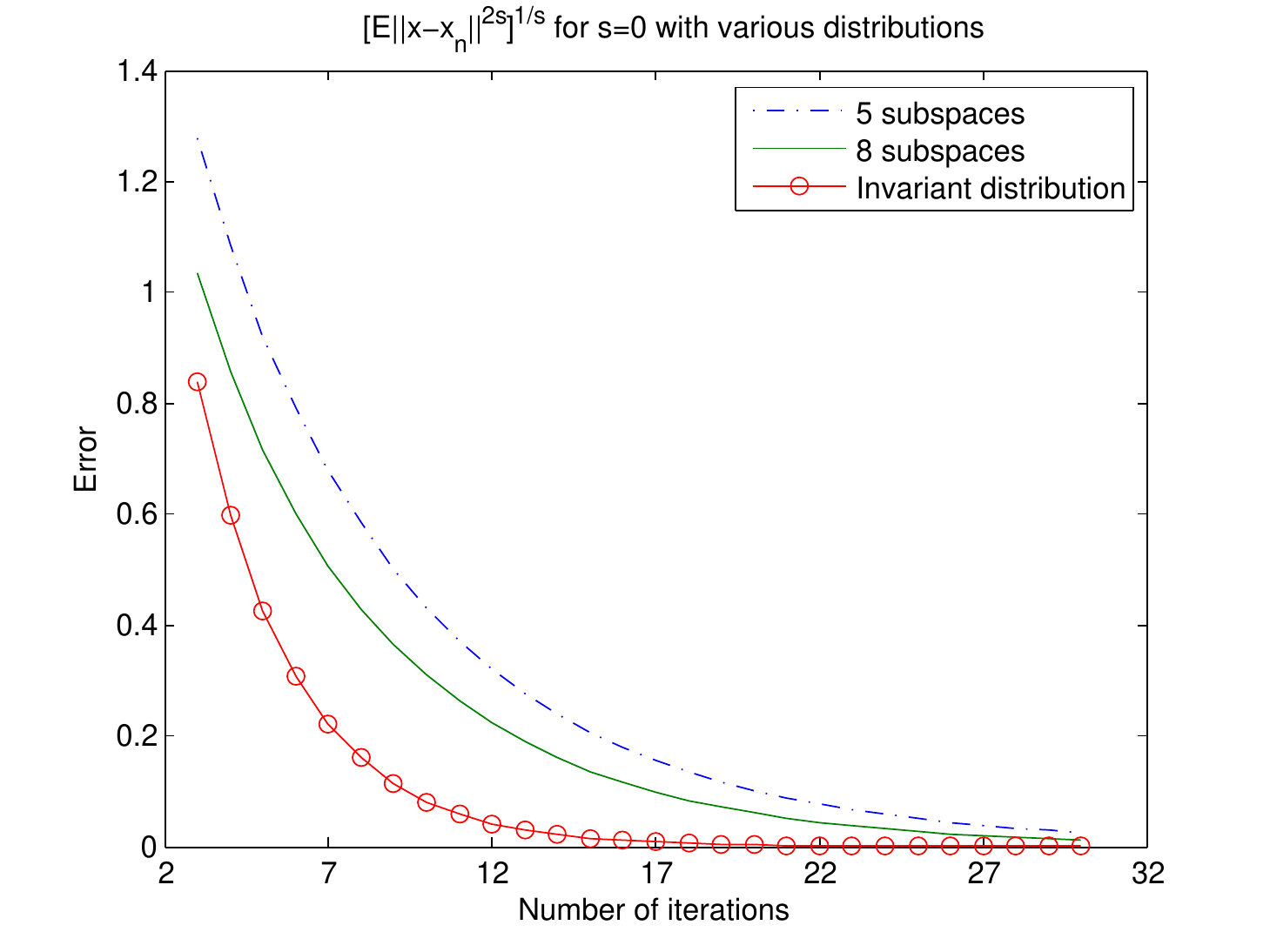}
                
                \label{fig:s0}
        \end{subfigure}
        \caption{Error moments for the invariant distribution $\Gamma_{2,5}$, the randomized 5 subspace distribution, and the randomized 8 subspace distribution, see Example \ref{exa:fds}.}\label{fig:fdis}
\end{figure}

\begin{figure}[b]
        \centering
        \begin{subfigure}[b]{0.33\textwidth}
                \centering
                \includegraphics[width=\textwidth]{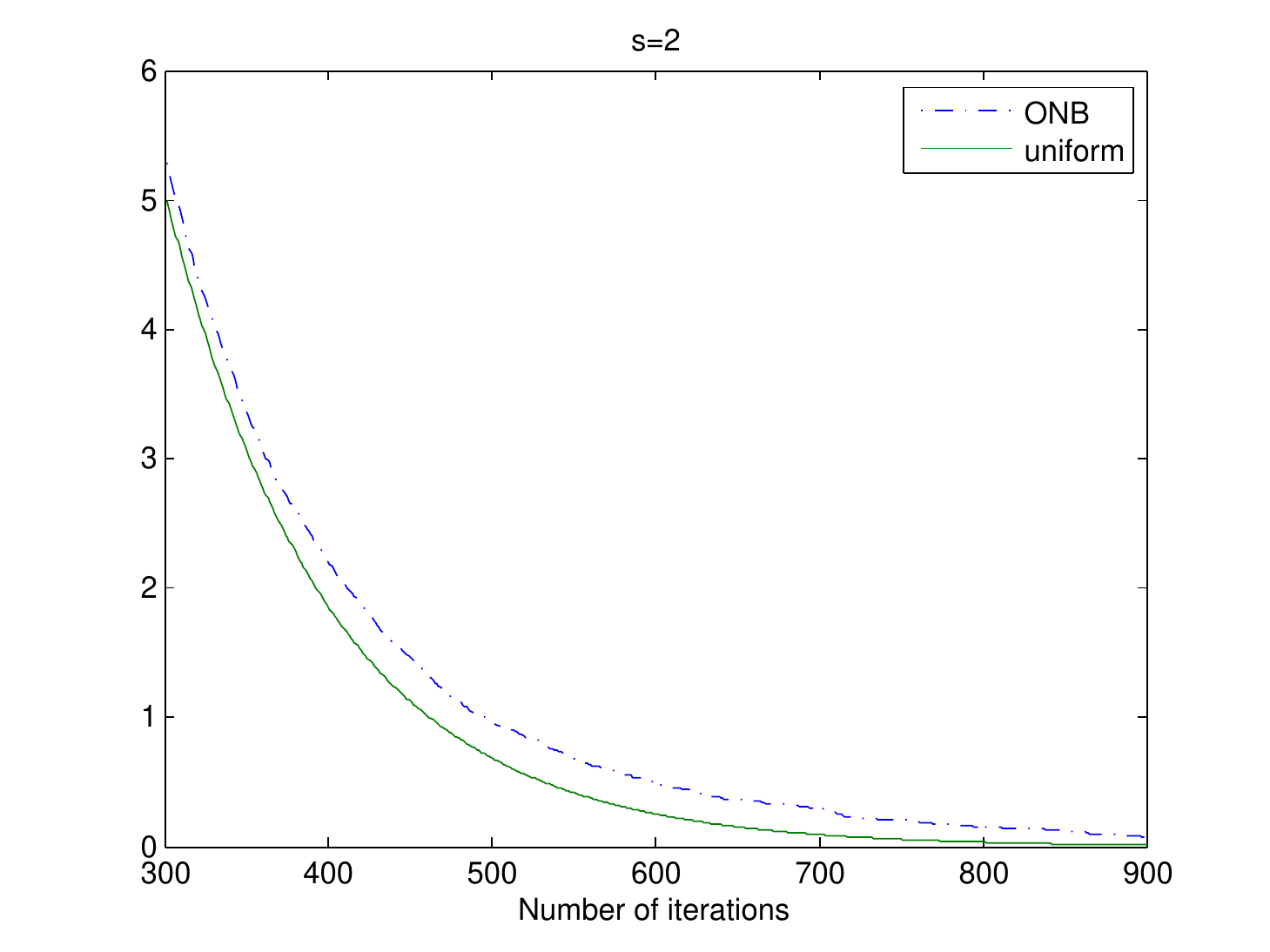}
                
                \label{fig:s2}
        \end{subfigure}
        ~\begin{subfigure}[b]{0.33\textwidth}
                \centering
                \includegraphics[width=\textwidth]{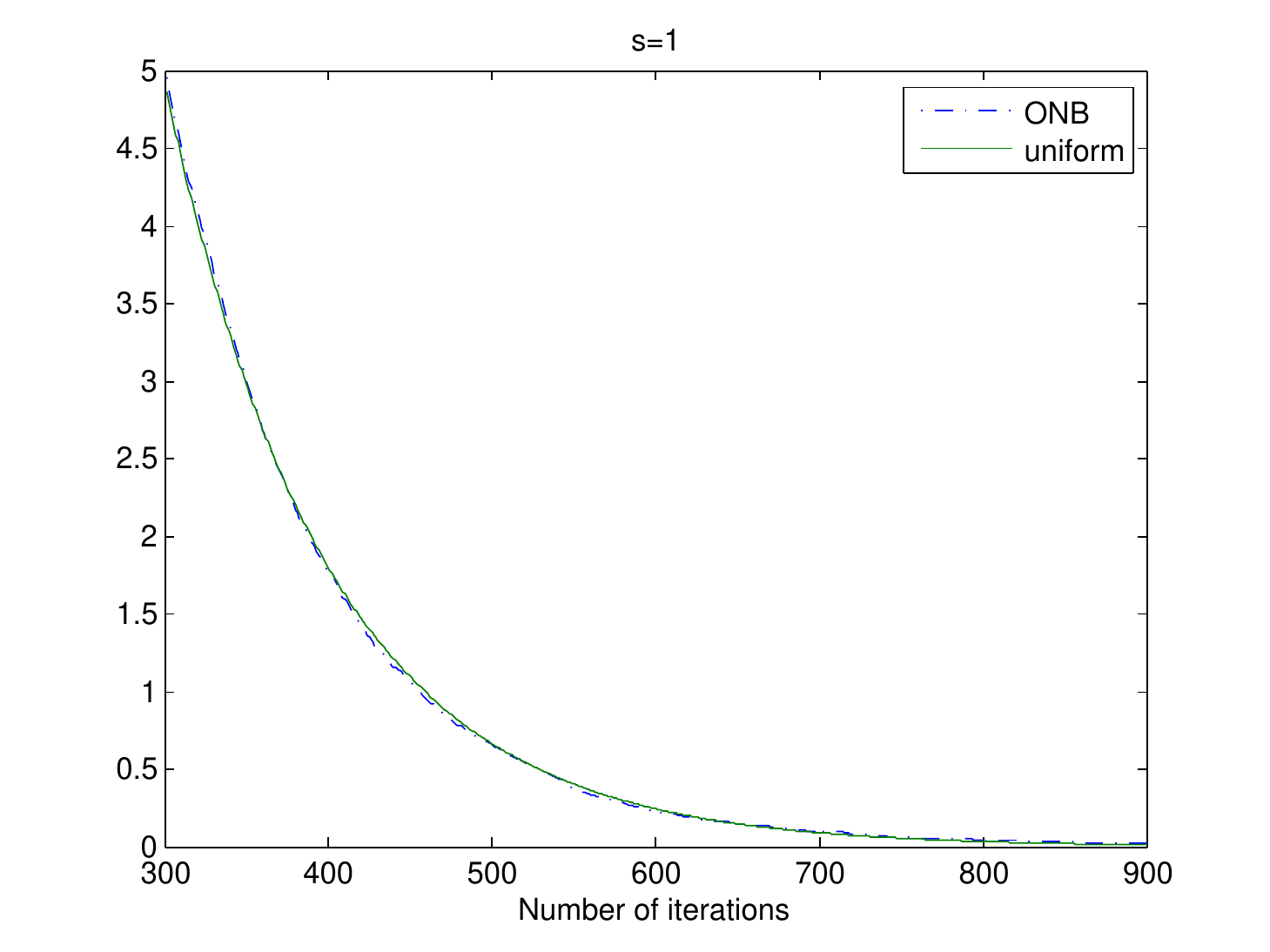}
                
                \label{fig:s1}
        \end{subfigure}
        ~\begin{subfigure}[b]{0.33\textwidth}
                \includegraphics[width=\textwidth]{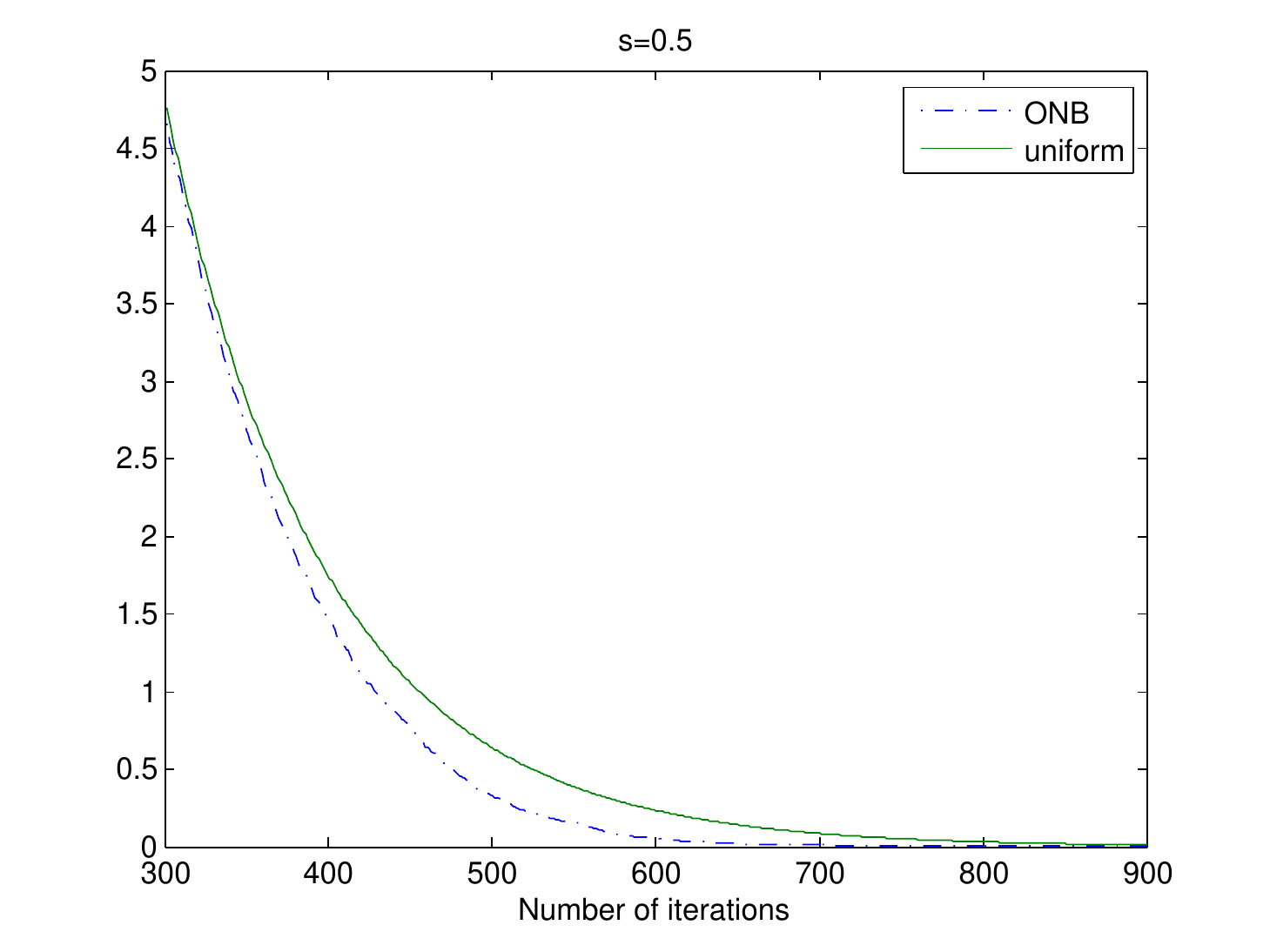}
         \end{subfigure}
       
        \caption{Error moments for the invariant distribution $\Gamma_{1,100}$ and the RONB distribution, see Example \ref{numexample-RONB}.}\label{fig:ONBU}
\end{figure}

\begin{figure}[b]
        \centering
        \begin{subfigure}[b]{0.33\textwidth}
                \centering
                \includegraphics[width=\textwidth]{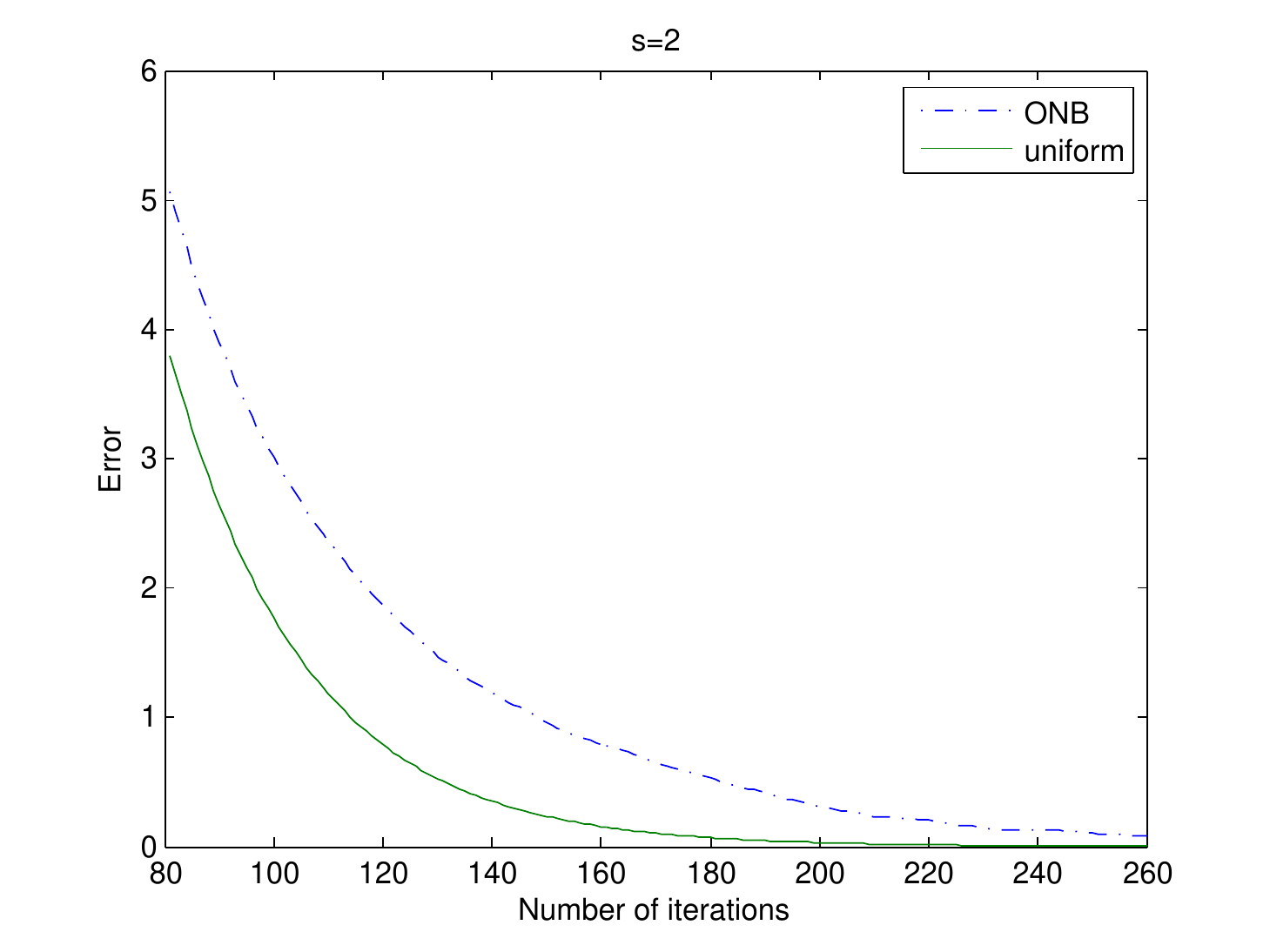}
                
                \label{fig:s2}
        \end{subfigure}
        ~\begin{subfigure}[b]{0.33\textwidth}
                \centering
                \includegraphics[width=\textwidth]{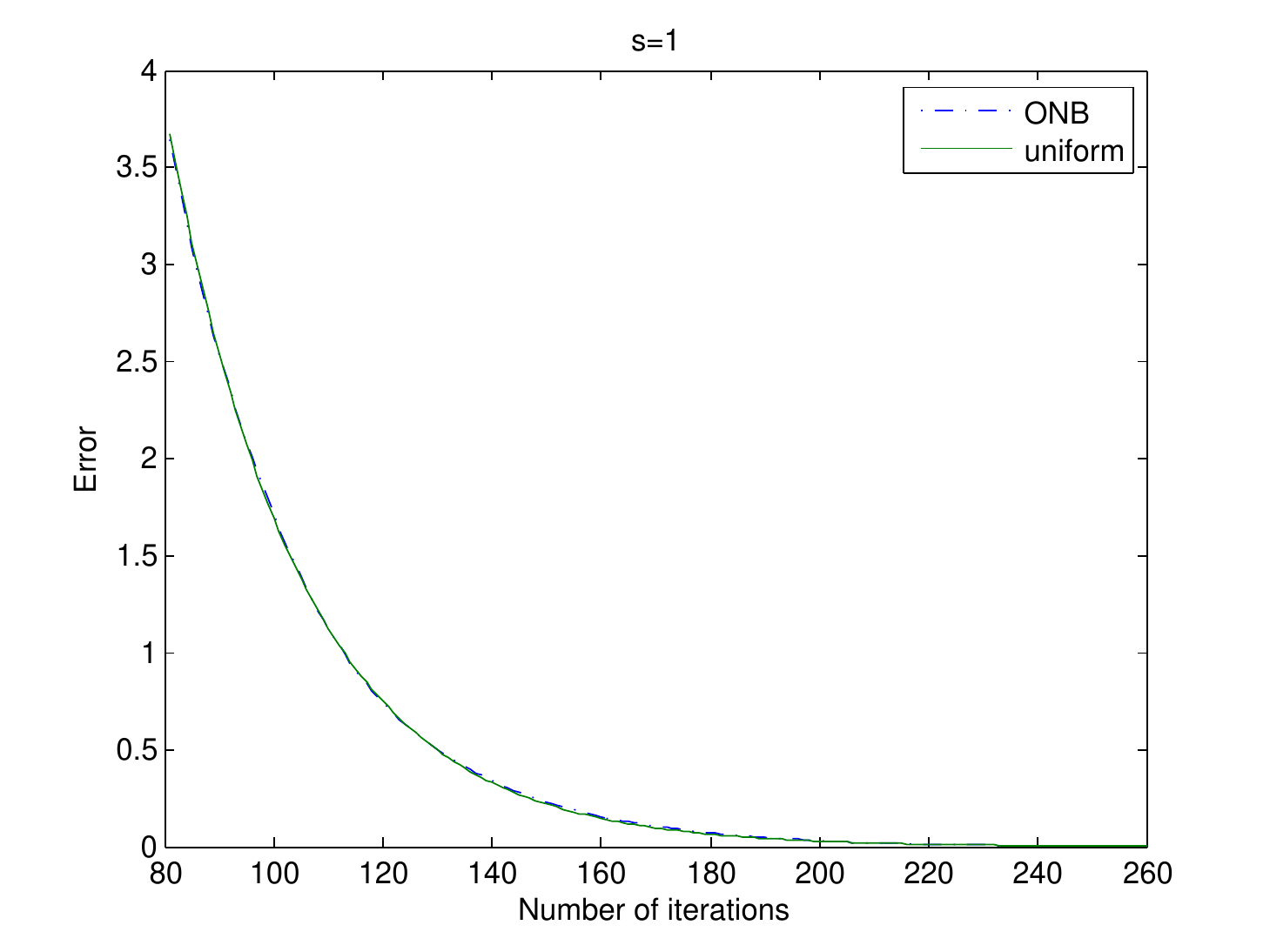}
                
                \label{fig:s1}
        \end{subfigure}
        ~\begin{subfigure}[b]{0.33\textwidth}
                \includegraphics[width=\textwidth]{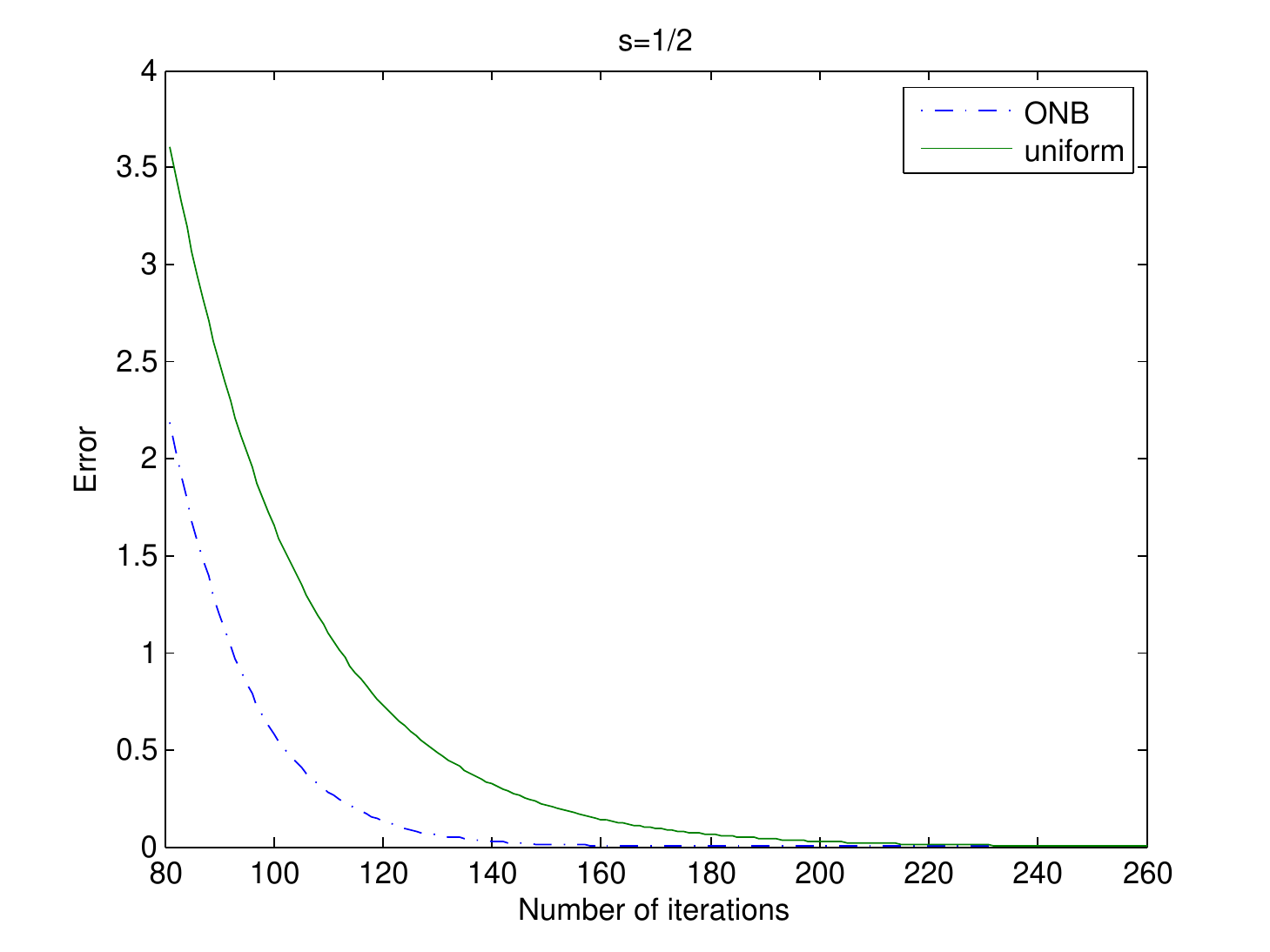}
         \end{subfigure}
       
        \caption{Error moments for the invariant distribution $\Gamma_{4,100}$ and the fusion frame ONB distribution, see Example \ref{exa:fonb}.}\label{fig:fONBU}
\end{figure}

\end{document}